\documentclass[11pt]{amsart}

\newcommand{\TheTitle}{On  open and closed convex codes}

\usepackage{amsthm}
\usepackage{graphicx,amsfonts, amsmath,  amssymb,color,amsfonts,color,url,subcaption, enumerate,xparse}
\usepackage{chngcntr}
\usepackage{tikz}
\usetikzlibrary{arrows}
\counterwithin{figure}{section}
\ExplSyntaxOn
\DeclareExpandableDocumentCommand{\eval}{m}{\int_eval:n {#1}}
\ExplSyntaxOff

\newcommand\norm[1]{\left\lVert#1\right\rVert}

\newcommand{\Cl}{\operatorname{cl}}
\newcommand{\Int}{\operatorname{int}}
\definecolor{darkgreen}{rgb}{0,0.5,0}

\newcommand{\od}{\stackrel{\mbox {\tiny {def}}}{=}}
\newcommand{\C}{\mathcal{C}}
\newcommand{\D}{\mathcal{D}}
\newcommand{\R}{\mathbb{R}}

\newcommand{\cU}{\mathcal{U}}
\newcommand{\U}{\mathcal{U}}

\newcommand{\cC}{\mathcal{C}}

\newcommand{\W}{\mathcal{W}}
\newcommand{\V}{\mathcal{V}}

\newcommand{\conv}{\mathrm{conv}}

\newcommand{\code}{\mathrm{code}}
\newcommand{\nerve}{\mathrm{nerve}}
\newcommand{\link}{\mathrm{link}}

\newcommand{\odim}{\mathrm{odim}\,}
\newcommand{\cdim}{\mathrm{cdim}\,}

 \definecolor{ttfftt}{rgb}{0.2,0.8,0.2}
\definecolor{qqqqff}{rgb}{0,0,1}
\definecolor{ffqqtt}{rgb}{1,0,0.2}
\definecolor{uququq}{rgb}{0.25,0.25,0.25}
\definecolor{zzttqq}{rgb}{0.6,0.2,0}
\definecolor{ffqqff}{rgb}{1,0,1}

\theoremstyle{definition}
\newtheorem{theorem}{Theorem}[section]
\newtheorem{proposition}[theorem]{Proposition}
\newtheorem{lemma}[theorem]{Lemma}

\newtheorem{definition}[theorem]{Definition}

\ExplSyntaxOn
\DeclareDocumentEnvironment{repetition}{m o}
{
  \tl_set_eq:Nc \l_tmpa_tl {#1}
  \IfNoValueTF{#2}{
    \tl_set:Nx \l_tmpb_tl {
      \tl_item:Nn \l_tmpa_tl {1}
      {\tl_item:Nn \l_tmpa_tl {2} }
      {}
      {\tl_item:Nn \l_tmpa_tl {4}}
    }
  }
  {
    \tl_set:Nx \l_tmpb_tl {
      \tl_item:Nn \l_tmpa_tl {1}
      {\tl_item:Nn \l_tmpa_tl {2} }
      {}
      {\tl_item:Nn \l_tmpa_tl {4} ~ \exp_not:N \ref{#2} }
    }
  }
  \tl_use:N \l_tmpb_tl
}
{
  \use:c {end #1}
}
\ExplSyntaxOff

\title{{\TheTitle}
}

\author{
  Joshua Cruz
  }
  \address{Department of Mathematics, Duke University}
  \email{joshua.cruz@duke.edu}
 \author{ Chad Giusti
 }
  \address{Departments of Bioengineering \& Mathematics, University of  Pennsylvania }
  \email{cgiusti@seas.upenn.edu}

 \author{ Vladimir Itskov
    }
     \address{Department of Mathematics, The Pennsylvania State University}
      \email{vladimir.itskov@math.psu.edu}
      
     \author{ Bill Kronholm}
         \address{Department of Mathematics, Whittier College}
	\email{wkronholm@whittier.edu}
	
\usepackage{amsopn}

\ExplSyntaxOn
\DeclareExpandableDocumentCommand{\eval}{m}{\int_eval:n {#1}}
\ExplSyntaxOff
\begin{document}
\maketitle

\begin{abstract} Neural codes serve as a language for neurons in the brain.  {\it Convex codes}, which arise from the pattern of intersections of convex sets in Euclidean space, are of particular relevance to neuroscience.
Not every code is convex, however, and the combinatorial properties of a code that determine its convexity are still poorly understood.  Here we find that a code that can be realized by a collection of open convex sets may or may not be realizable by closed convex sets, and vice versa, establishing that {\it open convex} and {\it closed convex} codes are distinct classes.  We also prove that {\it max intersection-complete} codes (i.e. codes that contain all intersections of maximal codewords) are both open convex and closed convex, and provide an upper bound for their minimal embedding dimension. Finally, we show that the addition of non-maximal codewords to an open convex code preserves convexity.
\end{abstract} 

\section{Introduction.}
 The brain represents information via patterns of neural activity.  Often, one can think of these patterns as strings of binary responses, where each neuron is ``on'' or ``off'' according to whether or not a given stimulus lies inside its receptive field.  In this scenario,  the {\it receptive field} $U_i \subset X$ of a neuron $i$ is simply the subset of stimuli to which it responds, with $X$ being the entire stimulus space.  A collection $\U = \{U_1,\ldots,U_n\}$ of receptive fields for a population of neurons $[n] \od \{1,\ldots,n\}$ gives rise to the {combinatorial code}\footnote{A {\it combinatorial code} is any collection of subsets $\C \subseteq 2^{[n]}$.  Each $\sigma \in \C$ is called a {\it codeword}.}
  \begin{equation*}  
\code(\U,X) \od \{ \sigma \subseteq [n] \text{ such that } \, A^{\U}_\sigma \neq \varnothing \} \subseteq 2^{[n]},
\end{equation*} 
where $2^{[n]}$ is the set of all subsets of $[n]$, and the {\it atoms} $A^{\U}_\sigma $ correspond to regions of the stimulus space carved out by $\U$:
\begin{equation*}  
A^{\U}_\sigma \od \left( \bigcap_{i \in \sigma} U_i \right) \setminus  \bigcup_{j \not \in \sigma} U_j \subseteq X.
\end{equation*} 
Here every stimulus $x \in A^{\U}_\sigma$ gives rise to the same neural response pattern, or {codeword}, $\sigma \subseteq [n]$.
By convention, $\cap_{i\in \varnothing}U_i=X$ and thus  $A^{\U}_\varnothing = X\setminus \left ( \bigcup _{i=1}^n U_i\right) $, so that $\varnothing  \in \code(\cU,X)$ if and only if $X \neq  \bigcup_{i=1}^n U_i$.  Note that  $\code(\U,X)$  may fail to be an abstract simplicial complex;  see e.g. Figure \ref{ex:codeexample}.

\medskip 
 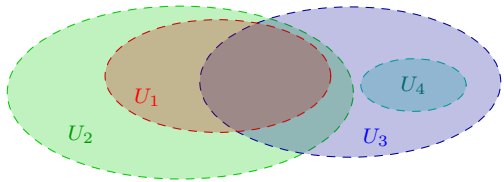
\begin{figure}[h]
\begin{center}
  \definecolor{qqccqq}{rgb}{0,0.8,0}
      \definecolor{qqaaqq}{rgb}{0,0.6,0}
    \definecolor{qqffqq}{rgb}{0,1,0}
\definecolor{qqqqzz}{rgb}{0,0,0.6}
\definecolor{rrqqqq}{rgb}{0.8,0,0}
\definecolor{ffqqqq}{rgb}{1,0,0}
\definecolor{qqqqaa}{rgb}{0,0,0.4}
\definecolor{qqzzzz}{rgb}{0,0.6,0.6}
\definecolor{qqrrrr}{rgb}{0,0.4,0.4}
\begin{tikzpicture}[line cap=round,line join=round,x=1.0cm,y=1.0cm, scale=1]
\draw [line width=0.4pt,dash pattern=on 3pt off 3pt,color=qqccqq,fill=qqccqq,fill opacity=0.25] (-2.0,2.9) ellipse (2.3cm and 1.15cm);
\draw [line width=0.4pt,dash pattern=on 3pt off 3pt,color=qqqqzz,fill=qqqqzz,fill opacity=0.25] (0.26,3.04) ellipse (2cm and 1cm);
\draw [line width=0.4pt,dash pattern=on 3pt off 3pt,color=rrqqqq,fill=rrqqqq,fill opacity=0.25] (-1.5,3.12) ellipse (1.5cm and .75cm);
\draw [line width=0.4pt,dash pattern=on 3pt off 3pt,color=qqzzzz,fill=qqzzzz,fill opacity=0.25] (1.1,3) ellipse (0.7cm and .35cm);
\begin{scriptsize}
\draw[color=ffqqqq] (-2.44,2.84) node {$U_1$};
\draw[color=qqqqff] (0.6,2.3) node {$U_3$};
\draw[color=qqaaqq] (-3.32,2.34) node {$U_2$};
\draw[color=qqrrrr] (1.1,3) node {$U_4$};
\end{scriptsize}
\end{tikzpicture}
\caption{An example of a cover $\U=\{U_i\} $ and its code, $\C= \code(\cU,X)  = \{\varnothing, 2, 3, 12, 23, 34, 123\}$, where $X = \mathbb R^2$.  Here we denote a codeword $\{i_1, i_2, \dots, i_k\} \in \cC$ by the string $i_1i_2\dots i_k$;  for example, $\{1,2,3\}$ is abbreviated to $123$.
Since $13 \not\in \cC$ but $13 \subset 123$, $\cC$ is not a simplicial complex.
}
\label{ex:codeexample}
\end{center}
\end{figure}


\begin{definition}\label{def:convex} 
We say that a combinatorial code $\C\subseteq 2^{[n]}$ is {\it open convex} if $\C = \code(\U,X)$ for a collection $\U=\{U_i\}_{i=1}^n$ of open convex subsets  $U_i  \subseteq X \subseteq \R^d$   for some $d\geq 1$.  Similarly, we say that $\C$ is {\it closed convex} if $\C = \code(\U,X)$  for a collection of  closed convex subsets  $U_i  \subseteq X \subseteq \R^d$.  
For an open convex code $\C$, the {\it embedding dimension} $\odim(\C)$ is the minimal $d$ for which there exists an open convex realization of $\C$ as $\code(\U,X)$.  Similarly, for a closed convex code $\cdim(\C)$ is the minimal $d$ that admits a closed convex realization of $\C$.
\end{definition}

Convex codes have special relevance to neuroscience because neurons in a number of areas of mammalian brains possess convex receptive fields.
A paradigmatic example is that of hippocampal {\it place cells} \cite{OKeefe1971}, a class of neurons in the hippocampus that act as position sensors.  Here the 
 relevant stimulus space $X\subset \mathbb{R}^d$ is the animal's environment, with $d\in \{1,2,3\}$ \cite{Yartsev367}. Receptive fields can be easily computed when both the neuronal activity data and the relevant stimulus space are available.  However, in many situations the relevant stimulus space for  a given neural population may be unknown.
This raises the natural question: how can one determine from the intrinsic properties of a combinatorial code whether or not it is an open (or closed) convex code?
What is the embedding dimension of a code -- that is, what is the dimension of the relevant stimulus space? How are open and closed convex codes related?  \\

The code of a cover  carries more information about the geometry/topology of the underlying space than the nerve of the cover.  For example, it imposes more constraints on the embedding dimension than what is imposed by the nerve \cite{neuralring}.  Arrangements of convex sets are ubiquitous in applied and computational topology, however all the standard constructions (e.g. the \v{C}ech complex) rely only on  the nerve of the cover, and do not carry any information about the arrangement beyond the nerve.    While the properties of nerves of convex covers were previously studied in \cite{kalaifvector1984,kalaifvector1986, Tancer-d-rep}, codes of convex covers are much less understood. Moreover, although any simplicial complex can be realized as the nerve of a convex cover (in high enough dimension), not all combinatorial codes can be realized from such convex set arrangements in Euclidean space. \\

 There is currently little understanding of what makes a code convex beyond `local obstructions'  to convexity \cite{GI13,Carina2015}.  Furthermore, local obstructions can only be used to show that a code is {\it not} convex, and the absence of local obstructions does not guarantee convexity of the code \cite{lienkaemper2015obstructions}.   To show that a code {\it is} convex, one must produce a convex realization, and there are few results 
  that guarantee such an open (or closed) convex realization exists.
Our first main result makes significant progress in this regard, as it provides a general  condition for determining that a code is convex from combinatorial properties alone.  Specifically, we show that {\it max intersection-complete} codes -- i.e., codes that contain all intersections of their maximal\footnote{A codeword in $\sigma \in \C$ is {\it maximal} if, as a subset $\sigma \subseteq [n]$, it is not contained in any other codeword of $\C$.} codewords -- are both open convex and closed convex. 
 
\begin{theorem}\label{t:mic}
Suppose $\C \subset 2^{[n]}$ is a max intersection-complete code. Then $\C$ is both open convex and closed convex.  Moreover, the embedding dimensions satisfy $\odim(\C) \leq \max\{2,(k-1)\}$ and $\cdim(\C) \leq \max\{2,(k-1)\}$, where $k$ is the number of maximal codewords of $\C$. 
\end{theorem}

The fact that max intersection-complete codes are open convex was first hypothesized in \cite{Carina2015}, where it was shown that these codes have no local obstructions.  In our proof we provide an explicit construction of the convex realizations and  the upper bound for the corresponding embedding dimensions.  Our next main result shows that open convex codes exhibit a certain type of  {\it monotonicity}, in the sense that adding non-maximal codewords to an open convex code preserves convexity.

\begin{theorem}\label{thm-monotone} Assume that a code $\C\subset 2^{[n]}$ is open convex.  If $\D \supset \C$ has the same maximal codewords as $\C$, then $\D$ is also open convex and has embedding dimension  $   \odim \D\leq \odim \C+1$.
\end{theorem}
\noindent It is currently unknown if the monotonicity property holds for closed convex codes. \\

Finally, we establish that open convex codes and closed convex codes are distinct classes.   This motivates us to define a non-degeneracy condition on the cover; we then show that this condition guarantees that the corresponding code is both open convex and closed convex (see Theorem \ref{thm-non-degenerate}, Section \ref{sec:open:vs:closed}).  This result suggests that combinatorial properties of convex codes are richer than originally believed.  We propose that codes that are both open convex and closed convex are the most relevant to neuroscience, as the intrinsic noise in neural responses \cite{Kandel} makes it unclear whether receptive fields should be considered to be open or closed. 
\section{Convex codes.} \label{sec:convex:codes}
We begin with  observing   that without sufficiently strong assumptions about the cover $\mathcal U=\{ U_i\}$, any code can be realized as $\code(\cU,X)$.

\begin{lemma} \label{lem:veryboring} Every code $\cC \subset 2^{[n]}$ can be obtained as $\cC=\code(\cU,X)$ for a collection of (not necessarily convex) $U_i\subset \R^1$.
\end{lemma}

\begin{proof} It suffices to consider the case where each $i \in [n]$ appears in some codeword $\sigma \in \cC$.
For each $\sigma \in \cC$, choose points $x_\sigma \in \R^1$ such that $x_\sigma \neq x_\tau$ if $\sigma \neq \tau$.
Define $U_i = \{x_\sigma \;|\; i \in \sigma\}$ and $\cU = \{U_i\}_{i\in[n]}$.
If $\varnothing \in \cC$, then $\cC = \code(\cU, \R^1)$.
Otherwise, $\cC=\code(\cU,X) $, where $X=\cup_{\sigma\in \C }\{x_\sigma\}$.
\end{proof}
\noindent The sets $U_i$ in the above proof are finite subsets of  $\R^1$.
However, even if one requires that the sets $U_i$ be open and connected, almost all codes  can still arise as the code of such cover. 

\begin{lemma}\label{lemma:openconnected} Any code $\cC \subset 2^{[n]}$ that contains all singleton codewords, i.e.\   $\forall i\in [n], \{i\}\in \cC$, can be obtained as $\cC=\code(\cU,X)$ for a collection of open connected subsets $U_i\subset \R^3$.
\end{lemma}
\begin{proof} Similar to the proof of Lemma \ref{lem:veryboring}, one can place disjoint open balls $B_\sigma\subset \mathbb R^3$ for each $\sigma\in \C$ and define 
$U_i =\left( \cup_{ i \in \sigma} B_\sigma\right) \cup T_i$, where each $T_i\subset \mathbb R^3$ is a collection of open ``narrow tubes'' that connect all the balls $B_\sigma$ with $\sigma\ni i$.
Because these sets are embedded in $\mathbb R^3$, the ``tubes'' $T_i$ can always be arranged so that for each $i\neq j$ the intersections $T_i\cap T_j$ are contained in the union of the balls $B_\sigma$.
By construction, these $U_i$ are connected and open and  $\cC=\code\left (\cU, \left(\cup_{i=1 }^n U_i\right) \cup B_\varnothing \right)$.
\end{proof}
\noindent The condition of having all singleton words can {\it not} be relaxed without any further assumptions. 
For example, it can be easily shown that the code $\cC = \{ \varnothing, 1, 2, 13, 23 \}$, previously described in \cite{neuralring,GI13} cannot be realized as a code of a cover by open connected sets\footnote{Indeed, assuming the converse, it follows that $U_3=\left(U_1\cap U_3\right) \cup \left(U_2\cap U_3\right) $ and, since this code does not contain a codeword $\sigma\supseteq 12$, we conclude that $U_1\cap U_2=\varnothing$ and $U_3$ is a union of two disjoint open sets, which yields a contradiction.}. \\

\subsection{Local obstructions to convexity.}

Any combinatorial code $\C\subset 2^{[n]}$ can be completed to an abstract simplicial complex $\Delta(\cC)$,  the  {\it simplicial complex of the code}, which is the minimal simplicial complex containing $\cC$.  Note that  $\Delta(\C)$ is determined solely by  the maximal codewords of $\C$ (facets  of $\Delta(\C)$).   A  code can thus be thought of as a simplicial complex with some of its non-maximal faces `missing'. Moreover,  given a collection of sets $\U$ and $X$,  one can easily see that the simplicial complex of $\code(\U,X)$ is equal to the usual {\it nerve} of the cover $\U$:
$$ \Delta\left( \code\left(\cU,X\right)\right )= 
\nerve(\cU)\od \{ 
\sigma \subseteq [n] \text{ such that} \, \bigcap_{i \in \sigma} U_i \neq \varnothing \}.
$$
For example, Figure \ref{ex:codeexample}  depicts a code of the form $\C = \code(\U,X)$ that differs from its simplicial complex $\Delta(\C)$ because the subset $\{1,3\}$ is missing. 
This results from the fact that $U_1\cap U_3 \subseteq U_2$, a set containment that is not encoded in $\nerve(\U)$.\\

Not every code arises from a closed convex or open convex cover.
For example, the code $\cC = \{ \varnothing, 1, 2, 13, 23 \}$  above cannot be an open (or closed) convex code.
The failure of this code to be convex is ``local'' in that it is  missing the codeword $3$,  and adding new codewords which do not include $i=3$ would not make this code convex.

\begin{definition}
For any $\sigma\subset [n]$ the {\it link} of $\cC$ at $\sigma$ is the code $\link_\sigma  \, \cC \subseteq 2^{[n]\setminus \sigma} \subset 2^{[n]}$ on the same set of neurons, defined as 
$$\link_\sigma \, \cC \od \left \{ \tau\, \vert\, \tau \cup \sigma \in \cC \text{ and } \tau\cap \sigma=\varnothing\right\}.$$
\end{definition}
\noindent Note that the link of a code is typically {\it not} a simplicial complex, but the simplicial complex of a link is the usual 
 $\link_\sigma \, \Delta =\{\nu\in \Delta \, \vert\, \, \nu\cup \sigma\in \Delta, \text{ and } \nu\cap \sigma =\varnothing \}$
of the appropriate simplicial complex:\footnote{This is because both $\link_\sigma \C$ and  $\link_\sigma \Delta$ have the same set of maximal elements.} 
\begin{equation*} \Delta\left(\link_\sigma \, \cC \right)=\link_\sigma\, \Delta\left (\cC\right).
\end{equation*}
Moreover, it is easy to see that if $\cC=\code(\cU,X)$, then for every non-empty $\sigma\in \Delta(\C)$ 
$$ \link_\sigma \, \C =\code \left( \left\{ U_j \cap U_\sigma\right\}_{j\in [n]\setminus \sigma}, U_\sigma\right),\quad \text{ where } U_\sigma=\bigcap_{i\in \sigma} U_i. $$
Since any intersection of convex sets is convex, we thus observe 
\begin{lemma}\label{lemma:convexinheritance}  If $\C$ is an open (or closed) convex code, then  for any $\sigma\in \Delta\left( \C\right) $,  $\link_\sigma\, \C$ is also an open (or closed) convex code. 
\end{lemma}

\noindent Note that for $\sigma \in \Delta(\C)$, 
\begin{equation}\label{eq:simplicialviolator}  \sigma \in \Delta(\C) \setminus \C  \, \iff\,  \varnothing \notin \link_\sigma\, \C.
\end{equation} 
We call  the  faces of $ \Delta(\C)$,  that are ``missing'' from the code,  {\it simplicial violators} of $\cC$.
If a code $\C$ is convex, and $\sigma$ is a simplicial violator, then the convex code $\link_\sigma\, \C=\code(\{V_i\},U_\sigma)$ is special  in that  the convex sets 
$ V_j= U_j \cap U_\sigma $ cover another  convex set $U_\sigma$ that is therefore contractible. The `local obstructions' to convexity arise from a special case of   the \emph{nerve lemma}.

\begin{lemma}[Nerve Lemma, \cite{Bjorner:1996, Edels2010}]
  For any finite cover $\mathcal V  = \{V_i\}_{i \in [n]}$  by convex sets $V_i\subset \mathbb R^d$ that are either all open or all closed\footnote{A formulation of the nerve lemma which applies to finite collections of closed, convex subsets of Euclidean space appears in \cite{Edels2010}, and follows from \cite[Theorem 10.7]{Bjorner:1996}.},  the abstract simplicial complex 
\begin{equation*}  
\nerve(\mathcal V)\od \{ 
\sigma \subseteq [n] \text{ such that} \, \bigcap_{i \in \sigma} V_i \neq \varnothing \} \subset 2^{[n]}, 
\end{equation*} 
known as the \emph{nerve} of the cover is homotopy equivalent to the underlying space $X=\cup_{i\in [n]}V_i$.
\end{lemma}

A simple corollary of Lemma \ref{lemma:convexinheritance} and the nerve lemma is the following observation (which first appeared in \cite{GI13}) that  provides a class of `local' obstructions to being an open (or closed) convex code.
\begin{proposition} \label{P:localobstruction} Let $\sigma\neq \varnothing $ be a simplicial violator of a code $\cC$.
If $\link_\sigma \, \Delta(\cC)$ is not a contractible simplicial complex, then $\cC$ is not an open (or closed) convex code.
\end{proposition}
\begin{proof} Assume the converse, i.e.\    $\cC$ is open (or closed) convex and $\sigma$ satisfies \eqref{eq:simplicialviolator}. Then the sets $U_j \cap U_\sigma$ cover a convex and  open (or closed) set $U_\sigma$, and thus by the nerve lemma the simplicial complex 
$$\nerve\left(  \left\{ U_j \cap U_\sigma\right\}_{j\in [n]\setminus \sigma}\right)=\Delta( \link_\sigma\, \C)= \link_\sigma\, \Delta(\C)$$
is contractible. 
\end{proof}

As an example, consider $\cC = \{\varnothing, 1, 2, 3, 4, 123, 124 \}$.
Then $\sigma = 12$ is a simplicial violator of $\cC$ and $\link_\sigma \, \cC = \{ 3 , 4 \}$.
Since $\Delta(\link_{\sigma} \, \cC)$ is not contractible, the code $\C$ is not the code of an open (or closed) convex cover.
This is perhaps the minimal example of a non-convex code that can be still realized by an open cover by connected sets\footnote{In fact, all the non-convex codes on three neurons (these were classified in \cite{neuralring}) cannot be realized by open (or closed) connected sets. This is because the only obstruction to convexity is the ``disconnection'' of one set, similar to the case of the code $\cC = \{ \varnothing, 1,2, 13, 23  \}$. }.\\

Note that if the condition that all sets are open, or alternatively all sets are closed, is dropped, then (at the time of this writing) there are no known obstructions for a code to arise as a code of a convex cover.
For instance, if one set is allowed to be of the ``wrong kind'', the code $\cC = \{ \varnothing, 1,2, 13, 23  \}$ above can be realized by intervals on a line, such as $U_1=(0,2)$ , $U_2=[2,4]$, $U_3=[1,3]$.
For this reason, we only consider either open or closed convex codes.

\subsection{Do  truly ``non-local''  obstructions via nerve lemma exist?}

The ``local'' obstructions to convexity in Proposition \ref{P:localobstruction} equally apply to any open (or closed) good cover, i.e.\   a cover where each non-empty intersection $U_\sigma=\cap_{i\in \sigma} U_i$ is contractible.  Since this property stems from applying the nerve lemma to the cover of $U_\sigma$ by the other contractible sets, it is natural to define a more general ``non-local'' obstruction to convexity that also stems from the nerve lemma. 
\begin{definition}
We say that a   non-empty  subset  $\sigma\subseteq [n]$ {\it covers}  a code $\C\subseteq 2^{[n]}$ if for every $\tau\in \C$,  
   $\tau\cap \sigma \neq \varnothing$. 
 \end{definition}
\noindent  Note that any code covered by at least one  non-empty set $\sigma$ does not have the  empty set. Moreover,   
$\sigma$   covers   $ \C=\code\left(\{U_i\}_{{i\in [n]} }, \bigcup_{i\in [n]} U_i\right)$ if and only if 
 $\bigcup_{i\in [n] }U_i=\bigcup_{j\in \sigma }U_j$.
\begin{lemma} \label{lemma:non-local} If there exist two  non-empty subsets $\sigma_1 ,\sigma_2 \subseteq [n]$, that both cover the code $\C\subseteq 2^{[n]}$, but  the codes $\C \cap \sigma_a \od \{ \tau\cap \sigma_a \vert \tau \in \C  \} \subseteq 2^{\sigma_a}$ for $a\in\{1,2\}$ 
have  simplicial complexes $\Delta\left(\C \cap \sigma_a \right) $  that 
are {\it not} homotopy equivalent, then $\cC$ is not  a code of a convex cover by  open (or closed) sets in $\mathbb R^d$. 
\end{lemma} 
\begin{proof} If such convex  cover existed, then the condition that each of the non-empty subsets $\sigma_a$ covers the code $\C$ implies that  that 
$ \cup_{i\in [n] }U_i=\cup_{j\in \sigma_a }U_j$ for each $a\in \{1,2\}$. Thus, by the nerve lemma, $ \Delta(\C)$ has the same homotopy type as   each of the complexes $\Delta\left(\C \cap \sigma_a\right) $. This yields a contradiction.
\end{proof}

The above obstruction to convexity can be thought as ``non-local'' because  it is conditioned on  the homotopy type  of a subset  that covers the entire code.   
While it is straightforward to produce combinatorial codes with these  ``non-local'' obstructions, we found that every  such code that we have considered\footnote{This included computer-assisted search among  random codes.}  inevitably  possesses a local obstruction for convexity. Perhaps the smallest such example is the code $\C=\{23, 14, 123 \} $ that meets the conditions of  Lemma \ref{lemma:non-local} with $\sigma_1=\{12\}$, and  $\sigma_2=\{34\}$, but  also has a local obstruction for the simplicial violator $\sigma=\{1\}$.  The exact reason for the significant difficulty of finding a  truly ``non-local'' obstruction  is still unclear.  Nevertheless, this provides some evidence for  the conjecture that any  code $\C \subset 2^{[n]} $ that has a ``non-local'' obstruction (i.e.\   the conditions of Lemma \ref{lemma:non-local} are met) must also have a ``local'' obstruction, i.e.\   a simplicial violator $\sigma\in \Delta(\C)\setminus \C$ such that $\Delta(\link_\sigma\, \cC)$ is not a contractible simplicial complex.
 
 \medskip 
\subsection{The difference between open and closed convex codes.}\label{sec:open:vs:closed}
The homotopy type obstructions via the nerve lemma are obstructions to being a code of a good cover (as opposed to convex sets) and  equally apply to both open and  closed versions of the Definition \ref{def:convex}.  However, it turns out that  the open and the closed convex codes are  distinct classes of codes.   Perhaps a minimal example of  an open convex code that is not closed convex is the code  
 \begin{equation}\label{eq:convexcounterexample}  \C= \{123, 126, 156, 456, 345, 234, 12, 16, 56, 45, 34, 23, \varnothing \}  \subseteq 2^{[6]}.
 \end{equation} 
 This code is realizable by  an open convex cover (Figure \ref{F:openconvexcover})  and  also by  an open or closed good cover (Figure \ref{F:opengoodcoverone}).  
  
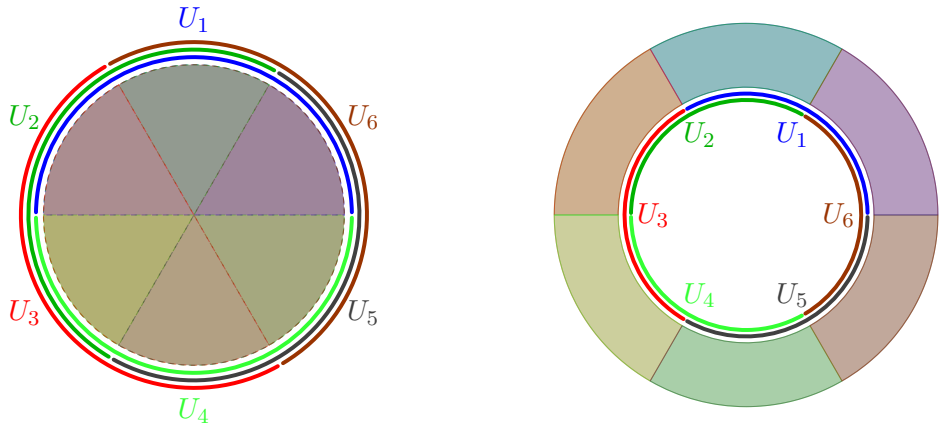
\begin{figure}[h!]\label{fig:closed-counterexample-variouscovers}
\begin{subfigure}{.4\textwidth}
\definecolor{qqccqq}{rgb}{0,0.7,0}
\definecolor{qqqqzz}{rgb}{0,0,0.6}
\definecolor{ffqqqq}{rgb}{1,0,0}
\definecolor{ttfftt}{rgb}{0.2,1,0.2}
\definecolor{qqqqff}{rgb}{0,0,1}
\definecolor{ffqqtt}{rgb}{1,0,0.2}
\definecolor{uququq}{rgb}{0.25,0.25,0.25}
\definecolor{zzttqq}{rgb}{0.6,0.2,0}
\definecolor{qqzzzz}{rgb}{0,0.6,0.6}
\definecolor{qqccqq}{rgb}{0,0.7,0}
\definecolor{qqqqzz}{rgb}{0,0,0.6}
\definecolor{ffqqqq}{rgb}{1,0,0}
\definecolor{ttfftt}{rgb}{0.2,1,0.2}
\definecolor{qqqqff}{rgb}{0,0,1}
\definecolor{ffqqtt}{rgb}{1,0,0.2}
\definecolor{uququq}{rgb}{0.25,0.25,0.25}
\definecolor{zzttqq}{rgb}{0.6,0.2,0}

\begin{center}
\begin{tikzpicture}[line cap=round,line join=round,x=1cm,y=1cm]
\filldraw[ttfftt, dash pattern = on 2pt off 2pt, opacity = 0.5, fill=ttfftt, fill opacity = 0.25] ({2*cos(180)},{2*sin(180)}) arc [radius=2, start angle=180, delta angle=180] -- cycle;
\draw[ttfftt, line width=1.5pt] ({2.1*cos(181)},{2.1*sin(181)}) arc [radius=2.1, start angle=181, delta angle=178];
\draw[color=ttfftt] ({2.6*cos(270)},{2.6*sin(270)}) node {$U_4$};
\filldraw[qqqqff, dash pattern = on 2pt off 2pt, opacity = 0.5, fill=qqqqff, fill opacity = 0.25] ({2*cos(0)},{2*sin(0)}) arc [radius=2, start angle=0, delta angle=180]                  -- cycle;
\draw[qqqqff, line width=1.5pt] ({2.1*cos(1)},{2.1*sin(1)}) arc [radius=2.1, start angle=1, delta angle=178];
\draw[color=qqqqff] ({2.6*cos(90)},{2.6*sin(90)}) node {$U_1$};
\filldraw[qqccqq, dash pattern = on 2pt off 2pt, opacity = 0.5, fill=qqccqq, fill opacity = 0.25] ({2*cos(60)},{2*sin(60)}) arc [radius=2, start angle=60, delta angle=180]  -- cycle;
\draw[qqccqq, line width=1.5pt] ({2.2*cos(61)},{2.2*sin(61)}) arc [radius=2.2, start angle=61, delta angle=178];
\draw[color=qqccqq] ({2.6*cos(150)},{2.6*sin(150)}) node {$U_2$};
\filldraw[ffqqqq, dash pattern = on 2pt off 2pt, opacity = 0.5, fill=ffqqqq, fill opacity = 0.25] ({2*cos(120)},{2*sin(120)}) arc [radius=2, start angle=120, delta angle=180] -- cycle;
\draw[ffqqqq, line width=1.5pt] ({2.3*cos(121)},{2.3*sin(121)}) arc [radius=2.3, start angle=121, delta angle=178];
\draw[color=ffqqqq] ({2.6*cos(210)},{2.6*sin(210)}) node {$U_3$};
\filldraw[uququq, dash pattern = on 2pt off 2pt, opacity = 0.5, fill=uququq, fill opacity = 0.25] ({2*cos(240)},{2*sin(240)}) arc [radius=2, start angle=240, delta angle=180]  -- cycle;
\draw[uququq, line width=1.5pt] ({2.2*cos(241)},{2.2*sin(241)}) arc [radius=2.2, start angle=241, delta angle=178];
\draw[color=uququq] ({2.6*cos(330)},{2.6*sin(330)}) node {$U_5$};
\filldraw[zzttqq, dash pattern = on 2pt off 2pt, opacity = 0.5, fill=zzttqq, fill opacity = 0.25] ({2*cos(300)},{2*sin(300)}) arc [radius=2, start angle=300, delta angle=180]                  -- cycle;
\draw[zzttqq, line width=1.5pt] ({2.3*cos(301)},{2.3*sin(301)}) arc [radius=2.3, start angle=301, delta angle=178];
\draw[color=zzttqq] ({2.6*cos(30)},{2.6*sin(30)}) node {$U_6$};
\end{tikzpicture}
\end{center}
\caption{An  open convex realization of $\C $}
\label{F:openconvexcover}
\end{subfigure} 
\hspace{0.02\textwidth}
\begin{subfigure}{.4\textwidth}
\definecolor{qqccqq}{rgb}{0,0.7,0}
\definecolor{qqqqzz}{rgb}{0,0,0.6}
\definecolor{ffqqqq}{rgb}{1,0,0}
\definecolor{ttfftt}{rgb}{0.2,1,0.2}
\definecolor{qqqqff}{rgb}{0,0,1}
\definecolor{ffqqtt}{rgb}{1,0,0.2}
\definecolor{uququq}{rgb}{0.25,0.25,0.25}
\definecolor{zzttqq}{rgb}{0.6,0.2,0}
\definecolor{qqzzzz}{rgb}{0,0.6,0.6}

\definecolor{qqccqq}{rgb}{0,0.7,0}
\definecolor{qqqqzz}{rgb}{0,0,0.6}
\definecolor{ffqqqq}{rgb}{1,0,0}
\definecolor{ttfftt}{rgb}{0.2,1,0.2}
\definecolor{qqqqff}{rgb}{0,0,1}
\definecolor{ffqqtt}{rgb}{1,0,0.2}
\definecolor{uququq}{rgb}{0.25,0.25,0.25}
\definecolor{zzttqq}{rgb}{0.6,0.2,0}
\definecolor{qqzzzz}{rgb}{0,0.6,0.6}
\begin{center}
\begin{tikzpicture}[line cap=round,line join=round,x=.85cm,y=.85cm]

\filldraw[qqqqff, opacity = 0.5, fill=qqqqff, fill opacity = 0.25] ({2*cos(0)},{2*sin(0)}) arc [radius=2, start angle=0, delta angle=120]
                  -- ({3*cos(120)},{3*sin(120)}) arc [radius=3, start angle=120, delta angle=-120]
                  -- cycle;
\draw[qqqqff, line width=1.5pt] ({1.9*cos(1)},{1.9*sin(1)}) arc [radius=1.9, start angle=1, delta angle=118];
\draw[color=qqqqff] ({1.45*cos(60)},{1.45*sin(60)}) node {$U_1$};

\filldraw[qqccqq, opacity = 0.5, fill=qqccqq, fill opacity = 0.25] ({2*cos(60)},{2*sin(60)}) arc [radius=2, start angle=60, delta angle=120]
                  -- ({3*cos(180)},{3*sin(180)}) arc [radius=3, start angle=180, delta angle=-120]
                  -- cycle;
\draw[qqccqq, line width=1.5pt] ({1.8*cos(61)},{1.8*sin(61)}) arc [radius=1.8, start angle=61, delta angle=118];
\draw[color=qqccqq] ({1.45*cos(120)},{1.45*sin(120)}) node {$U_2$};

\filldraw[ffqqqq, opacity = 0.5, fill=ffqqqq, fill opacity = 0.25] ({2*cos(120)},{2*sin(120)}) arc [radius=2, start angle=120, delta angle=120]
                  -- ({3*cos(240)},{3*sin(240)}) arc [radius=3, start angle=240, delta angle=-120]
                  -- cycle;
\draw[ffqqqq, line width=1.5pt] ({1.9*cos(121)},{1.9*sin(121)}) arc [radius=1.9, start angle=121, delta angle=118];
\draw[color=ffqqqq] ({1.45*cos(180)},{1.45*sin(180)}) node {$U_3$};

\filldraw[ttfftt, opacity = 0.5, fill=ttfftt, fill opacity = 0.25] ({2*cos(180)},{2*sin(180)}) arc [radius=2, start angle=180, delta angle=120]
                  -- ({3*cos(300)},{3*sin(300)}) arc [radius=3, start angle=300, delta angle=-120]
                  -- cycle;
\draw[ttfftt, line width=1.5pt] ({1.8*cos(181)},{1.8*sin(181)}) arc [radius=1.8, start angle=181, delta angle=118];
\draw[color=ttfftt] ({1.45*cos(240)},{1.45*sin(240)}) node {$U_4$};

\filldraw[uququq, opacity = 0.5, fill=uququq, fill opacity = 0.25] ({2*cos(240)},{2*sin(240)}) arc [radius=2, start angle=240, delta angle=120]
                  -- ({3*cos(0)},{3*sin(0)}) arc [radius=3, start angle=0, delta angle=-120]
                  -- cycle;

\draw[uququq, line width=1.5pt] ({1.9*cos(241)},{1.9*sin(241)}) arc [radius=1.9, start angle=241, delta angle=118];
\draw[color=uququq] ({1.45*cos(300)},{1.45*sin(300)}) node {$U_5$};

\filldraw[zzttqq, opacity = 0.5, fill=zzttqq, fill opacity = 0.25] ({2*cos(300)},{2*sin(300)}) arc [radius=2, start angle=300, delta angle=120]
                  -- ({3*cos(60)},{3*sin(60)}) arc [radius=3, start angle=60, delta angle=-120]
                  -- cycle;

\draw[zzttqq, line width=1.5pt] ({1.8*cos(301)},{1.8*sin(301)}) arc [radius=1.8, start angle=301, delta angle=118];
\draw[color=zzttqq] ({1.45*cos(0)},{1.45*sin(0)}) node {$U_6$};
\end{tikzpicture}
\end{center}
\caption{A closed good cover realization of $\C $}
\label{F:opengoodcoverone}
 \end{subfigure}
 \caption{Two different realizations of the code $\C$ in \eqref{eq:convexcounterexample}. In both realizations, each set $U_i$ is covered by the others, and is indicated by the colored arcs external to the sets; the colors of regions are combinations of the colors of the constituent sets. For example, in (a), $U_1$ is the open upper half-disk, while in (b) $U_1$ is the top right closed annular section.}
\end{figure}
 \begin{lemma} \label{lemma:example:nonconvex}The code \eqref{eq:convexcounterexample} is not closed convex.
 \end{lemma} 
 \noindent  The proof is given in the Appendix, Section \ref{subsec5.1}. 
 A different  example, 
\begin{equation}\label{eq:Anne:example}
\C = \{2345, 124, 135, 145, 14, 15, 24, 35, 45, 4, 5\} \subseteq 2^{[5]}, 
\end{equation}
  was originally considered in \cite{lienkaemper2015obstructions}, where it was proved that it is not open convex and possesses a  realization by a good open cover (Figure \ref{F:opengoodcover}), thus does not have any 
``local obstructions'' to convexity. 
 However, it turns out that this code is  closed convex (see a closed realization in Figure \ref{F:closedconvexcover}). 
 
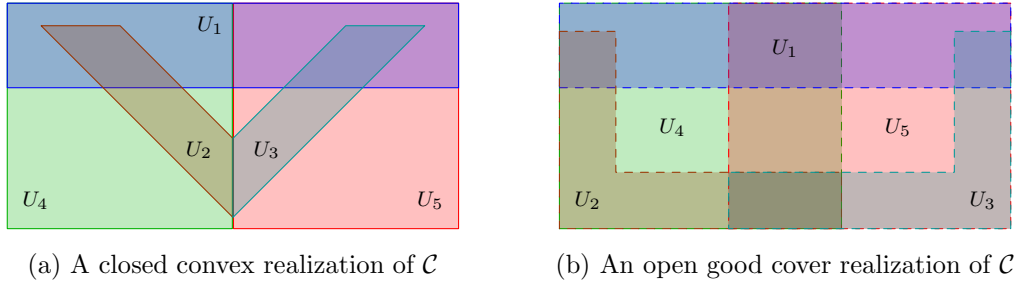
\begin{figure}[h!]\label{fig:variouscovers}
\begin{subfigure}{.4\textwidth}
\definecolor{qqccqq}{rgb}{0,0.7,0}
\definecolor{qqqqzz}{rgb}{0,0,0.6}
\definecolor{ffqqqq}{rgb}{1,0,0}
\definecolor{ttfftt}{rgb}{0.2,1,0.2}
\definecolor{qqqqff}{rgb}{0,0,1}
\definecolor{ffqqtt}{rgb}{1,0,0.2}
\definecolor{uququq}{rgb}{0.25,0.25,0.25}
\definecolor{zzttqq}{rgb}{0.6,0.2,0}
\definecolor{qqzzzz}{rgb}{0,0.6,0.6}
\begin{center}
\begin{tikzpicture}[line cap=round,line join=round,>=triangle 45,x=1.5cm,y=1.5cm]
\draw [fill=qqccqq, fill opacity=0.25,qqccqq]  (-2,-2) rectangle (-0.005,0);
\draw [fill=ffqqqq, fill opacity=0.25,ffqqqq]  (.005,-2) rectangle (2,0);
\draw [fill=qqqqff, fill opacity=0.25,qqqqff]  (-2,-0.75) rectangle (2,0);
\draw [fill=zzttqq, fill opacity = 0.25, zzttqq] (-1.7, -0.2) -- (-1, -0.2) -- (0, -1.2) -- (0, -1.9) -- cycle;
\draw [fill=qqzzzz, fill opacity = 0.25, qqzzzz] (1.7, -0.2) -- (1, -0.2) -- (0, -1.2) -- (0, -1.9) -- cycle;
\begin{scriptsize}
\draw[color=black] (-0.2,-0.2) node {$U_1$};
\draw[color=black] (-0.3,-1.3) node {$U_2$};
\draw[color=black] (-1.75,-1.75) node {$U_4$};
\draw[color=black] (1.75,-1.75) node {$U_5$};
\draw[color=black] (0.3,-1.3) node {$U_3$};
\end{scriptsize}
\end{tikzpicture}
\end{center}
\caption{A closed convex realization of $\C $}
\label{F:closedconvexcover}
\end{subfigure} 
\hspace{0.02\textwidth}
\begin{subfigure}{.4\textwidth}
\definecolor{qqccqq}{rgb}{0,0.7,0}
\definecolor{qqqqzz}{rgb}{0,0,0.6}
\definecolor{ffqqqq}{rgb}{1,0,0}
\definecolor{ttfftt}{rgb}{0.2,1,0.2}
\definecolor{qqqqff}{rgb}{0,0,1}
\definecolor{ffqqtt}{rgb}{1,0,0.2}
\definecolor{uququq}{rgb}{0.25,0.25,0.25}
\definecolor{zzttqq}{rgb}{0.6,0.2,0}
\definecolor{qqzzzz}{rgb}{0,0.6,0.6}
\begin{center}
\begin{tikzpicture}[line cap=round,line join=round,>=triangle 45,x=1.5cm,y=1.5cm]
\draw [dash pattern=on 3pt off 3pt,fill=qqccqq, fill opacity=0.25,qqccqq]  (-2,-2) rectangle (0.5,0);
\draw [dash pattern=on 3pt off 3pt,fill=ffqqqq, fill opacity=0.25,ffqqqq]  (-0.5,-2) rectangle (2,0);
\draw [dash pattern=on 3pt off 3pt,fill=qqqqff, fill opacity=0.25,qqqqff]  (-2,0) rectangle (2,-0.75);
\draw [dash pattern=on 3pt off 3pt,fill=zzttqq, fill opacity = 0.25, zzttqq] (-2, -0.25) -- (-1.5, -0.25) -- (-1.5,-1.5) -- (0.5, -1.5) -- (0.5,-2) -- (-2, -2) -- cycle;
\draw [dash pattern=on 3pt off 3pt,fill=qqzzzz, fill opacity = 0.25, qqzzzz] 
(2, -0.25) -- (1.5, -0.25) -- (1.5,-1.5) -- (-0.5, -1.5) -- (-0.5,-2) -- (2, -2) -- cycle;

\begin{scriptsize}
\draw[color=black] (-0,-0.4) node {$U_1$};
\draw[color=black] (-1.75,-1.75) node {$U_2$};
\draw[color=black] (-1,-1.1) node {$U_4$};
\draw[color=black] (1,-1.1) node {$U_5$};
\draw[color=black] (1.75,-1.75) node {$U_3$};
\end{scriptsize}
\end{tikzpicture}
\end{center}
\caption{An open good cover realization of $\C $}
\label{F:opengoodcover}
 \end{subfigure}
 \caption{Two different realizations of the code in \eqref{eq:Anne:example}.}
\end{figure}
\medskip 

  The examples in \eqref{eq:convexcounterexample} and \eqref{eq:Anne:example} show that 
open convex and closed convex are distinct classes of codes. Moreover, they illustrate that one cannot generally ``convert''  an open convex realization   into  a closed convex realization  or vice versa by simply taking closures  or interiors of sets in a  cover. Nevertheless, it is  intuitive that open and closed versions of a ``sufficiently non-degenerate''  cover should yield  the same code.  
 
A natural candidate for such a condition would be that the sets in the cover $\U$ are {\it in general position}, i.e. there exists $\varepsilon>0$ such that any cover $\V=\{V_i\}$ whose sets  $V_i$ are no further than $\varepsilon$ from  $U_i$ in the Hausdorff distance\footnote{Recall that the Hausdorff distance between two subsets $U$ and $V$ of a Euclidean space is defined as $$d_{H}(U,V) = \max \{\,\sup_{x \in U} \{  \inf_{y \in V} \norm{x-y}\},\, \sup_{y \in V}  \{ \inf_{x \in U} \norm{x-y} \}\, \}.$$}, has the same code: $ \code(\U,\mathbb R^d)= \code(\V,\mathbb R^d)$. However, being in general position is too strong a condition. This is because there are  covers of interest (such as those in Section \ref{sec:maxint}) that are not in general position yet yield the same code after taking the closure or interior. We therefore consider the following  weaker condition.  

\begin{definition}\label{dfn:nondegenerate} A cover $\mathcal{U} =\{U_i\}_{i\in [n] }$, with   $ U_i\subseteq \mathbb R^d$,  is {\it non-degenerate} 
if the following two conditions hold:
\begin{itemize}
\item[(i)] For all $\sigma \in \code(\mathcal{U},\mathbb R^d )$, the atoms $A^{\U}_\sigma$  are  {\it top-dimensional}, i.e.\   any non-empty intersection
 with an open set $B\subseteq \mathbb R^d $ has non-empty interior: 
 \begin{equation*} \label{eq:ndg}
B \text{ is open and }   A^{\U}_\sigma\cap B \neq \varnothing  \implies  \Int(A^{\U}_\sigma\cap B) \neq \varnothing.
  \end{equation*}
  \item[(ii)] For all non-empty $\sigma \subseteq [n]$, $\bigcap_{i\in \sigma}\partial U_i \subseteq \partial \left(\bigcap_{i\in \sigma} U_i\right)$.
  \end{itemize}
  \end{definition}
  \noindent  Note that if a cover $\U$ is open,  convex and in general position, then it is non-degenerate (see Lemma \ref{lemma:general:postiotion} in the Appendix), while  an open  convex and non-degenerate cover need not be in general position.  
 We should also note that the two seemingly  separate conditions (i) and (ii) in the above  definition are motivated by the following observation.
  \begin{lemma}\label{lemma:nondeg} Assume that  $\mathcal{U} =\{U_i\}$ is a  finite cover by convex sets. Then, 
\begin{itemize}
\item[]  if all $U_i$ are open and $\mathcal{U}$ satisfies the condition (ii), then  it also satisfies the condition (i); 
\item[]  if all $U_i$ are closed and $\mathcal{U}$ satisfies the condition (i), then  it also satisfies the condition (ii). 
\end{itemize}
 \end{lemma} 
\noindent  The proof is given in the Appendix (Section \ref{sec:nondeg:condition:lemmas}, Lemmas \ref{L:open_nondegen} and \ref{lemma:key:closed:lemma}).  Note that  if the sets  $U_i$ are {\it open}  and convex, then condition (i) does not imply condition (ii), similarly  if the sets  $U_i$ are {\it closed}  and convex
then condition (ii) does not imply condition (i)\footnote{ For example, the  cover 
by the  open convex sets  $U_1= \{ (x,y)\in \mathbb R^2 \,\vert\,   y>x^2\}$ and  $U_2=\{ (x,y)\in \mathbb R^2 \,\vert\,   y< -x^2\}$ satisfies condition (i), but does not satisfy condition (ii). Similarly, the closed subsets of the real line, $U_1=\{x\leq 0\}$, $U_2=\{x\geq 0\}$  satisfy condition (ii), but   do not satisfy  condition (i).}.

 For an open cover  $\mathcal{U} = \{U_i\}$, we denote by $\text{cl} (\cU) $ the cover by the closures 
$V_i=\text{cl}(U_i)$. Similarly, for a closed cover $\mathcal{U} = \{U_i\}$ we denote by $\Int(\mathcal U)$ the cover by the interiors $V_i=\Int (U_i)$. Recall  that if a set is convex, then both its  closure and its interior are convex. 

\begin{theorem}\label{thm-non-degenerate} Assume that $\mathcal{U} = \{U_i\}$ is a convex and non-degenerate cover, then 
\begin{eqnarray*}  U_i \text{ are open } \implies &     \code(\cU,\mathbb R^d) =\code\left(\text{cl}\left(\cU\right),\mathbb R^d\right);\\
      U_i \text{ are closed } \implies &   \code(\cU,\mathbb R^d) =\code\left(\Int\left (\cU\right ),\mathbb R^d \right).
\end{eqnarray*} 
\end{theorem}
\noindent The proof is given in the Appendix (Section \ref{sec:nondeg:proofs}). 
This theorem guarantees that  if an open convex code  is realizable by a  non-degenerate cover, then it is also closed convex; similarly  if a closed convex code  is realizable by a  non-degenerate cover, then it is also open convex. Non-degenerate covers are thus  {\it natural}  in  the neuroscience context, where receptive fields (i.e.\   the sets $U_i$)   should not change their code after  taking closure or interior, since such changes in code would be undetectable in the presence of standard neuronal noise.  This suggests that convex codes  that arise from non-degenerate covers should serve as the  standard model  for  convex codes in neuroscience-related contexts.  Note that  the existence of a non-degenerate convex cover realization  is {\it extrinsic}  in that it is not defined in terms of  the combinatorics of the code alone. A combinatorial description of such codes is unknown at the time of this writing.

\bigskip 
\section{Monotonicity of open convex codes.}\label{sec:Monotonicity} 
The set of all codes $\C\subseteq 2^{[n]}$ with a prescribed simplicial complex $K=\Delta(C)$  forms a poset.  
It is easy to see that if $\C$ is a convex code then its sub-codes  can be non-convex. For example any non-convex code is a  sub-code of  its simplicial complex, and  every simplicial complex is both an open and closed convex code (this follows from  Theorem \ref{t:mic} in Section \ref{sec:maxint}).  
It turns out that open convexity is a monotone increasing property. \\

\noindent {\bf Theorem~\ref{thm-monotone}.} 
Assume that a code $\C\subset 2^{[n]}$ is open convex.  Then every code $\D$ that satisfies $\C\subsetneq \D\subseteq \Delta(\C)$ is also open convex with open embedding dimension  $   \odim \D\leq \odim \C+1$.\\

Note that the above bound on  the embedding dimension is sharp. For example, the open convex code $\C=\{ 123, 12, 1 \}$ has embedding dimension $\odim \C=1$, but its simplicial complex $\D=\Delta(\C)$ has embedding dimension  $\odim \D=2$.  To prove this theorem we shall use the following lemma. Let $M(\C)$ denote the facets of the simplicial complex $\Delta(\C)$. 
\begin{lemma} \label{lem:Ball:rules} Let $\U = \{U_i\}$ be an open convex cover in $\mathbb R^d$, $d\geq 2$,  with  $\C=\code\left(\mathcal U,X\right)$. Assume that there exists an 
open Euclidean ball $B\subset \mathbb R^d $ such that $\code(\{B\cap U_i\},B\cap X) = \C$, and for every maximal set  $\alpha \in M(\C)$, its atom has non-empty intersection with the $(d-1)$-sphere: $\partial B \cap   A^\U_\alpha \neq \varnothing $. 
 Then for every $\D$ such that     $\C  \subsetneq \D\subseteq \Delta(\C) $, 
there exists an open convex cover $\V =\{V_i\}$ with $V_i\subseteq U_i$, such that $\D= \code(\mathcal V, B\cap X )$.  Moreover, if the cover $\U  $ is non-degenerate, then the cover $\V$ can also be chosen to be non-degenerate. 
\end{lemma} 
The proof of this lemma is given in Section \ref{apendix:monotone}.  Intuitively, the reason why this lemma holds is that one can ``chip away'' small pieces from  the ball $B$ inside some  atoms $A^\U_\alpha$ to uncover only the atoms corresponding to the codewords in  $\D\setminus \C$.   
\begin{proof}[Proof of Theorem \ref{thm-monotone}] 
Assume that $\U$ is an open convex cover  in $\R^d$ with $\C=\code(\U,X)$. Since there are only finitely many codewords, there exists a radius $r>0$ and  an open Euclidean ball $B^d_r \subset \mathbb R^d$, of radius $r$,  centered at the origin,  that  satisfy $\code(\{B^d_r\cap U_i\},B^d_r   \cap X) = \C$.  Let $\pi: \R^{d+1} \to \R^d$ be the standard projection. Let $B\od B^{d+1}_r$ denote the open ball in $\mathbb R^{d+1}$, centered at the origin and  of the same radius $r$.  Define $\tilde U_i = \pi^{-1}(U_i)$. By construction,  $\tilde{\U} =\{\tilde U_i \}$ is an open, convex cover, such that each of its atoms has non-empty intersection with the sphere $\partial B$.  
Moreover, $\code(\{B\cap \tilde U_i\},B\cap \pi^{-1}(X)) = \C$. 
Thus the conditions of Lemma \ref{lem:Ball:rules}  are satisfied for  the cover $\tilde \U$, and $\D$ is an open convex code with $   \odim \D\leq \odim \C+1$.
\end{proof}

Note that the proof of Lemma \ref{lem:Ball:rules} (see Section \ref{apendix:monotone}) breaks down if one assumes that the convex sets $U_i$ are closed. Moreover,  it is currently not known if the monotonicity property  holds in the setting of the closed convex codes. The differences between the open convex and the closed convex codes (described in the previous section)  leave enough room for either possibility. 

\bigskip 
\section{Max intersection-complete codes are open and closed convex.} \label{sec:maxint}
Here we introduce max intersection-complete codes and prove that they are open convex and closed convex. The open convexity of max 
intersection-complete codes was first hypothesized in \cite{Carina2015}.

\begin{definition} 
The {\it intersection completion of a code $\cC$} is the code that consists of all non-empty intersections of codewords in $\cC$: 
\begin{equation*}  \widehat{\cC} = \{ \sigma \;|\; \sigma = \bigcap_{\nu \in  \C^\prime}\nu \text{ for some non-empty subcode } \C^\prime   \subseteq \C\}.\end{equation*} 
\end{definition}
\noindent Note that the intersection completion satisfies $\cC\subseteq \widehat{\cC}\subseteq \Delta(\cC)$.

\begin{definition} Let $\C \subset 2^{[n]}$ be a code,
and denote by $M(\C) \subset \C$ the subcode consisting of all maximal codewords\footnote{Equavalently, the facets of $\Delta(\C)$.} of $\C$. A code $\cC$ is said to be \\

\begin{itemize}
\item {\it intersection-complete}  if $ \widehat{\cC}=\cC$; \\
\item  {\it max intersection-complete} if  $ \widehat{M(\cC)}\subseteq \cC$.
\end{itemize}
\end{definition}

\medskip 
\noindent Note that any simplicial complex (i.e.\   $\cC=\Delta(\cC)$) is intersection-complete and  any intersection-complete code  is  max intersection-complete. 
Intersection-complete codes allow a  simple construction of a closed convex realization that we describe  in  Section \ref{apendix:IC}  (see Lemma \ref{l:pot_ic}). However,  in order to prove that max intersection-complete codes are both open and closed convex, we need the following. 

\begin{proposition}\label{T:max_code}  Let $\C \subset 2^{[n]}$ be a code with $k$ maximal elements. 
Then  there exists an open convex and non-degenerate cover $\cU $ in $d=(k-1)$-dimensional space whose code is the intersection completion of the maximal elements in $\C$:  $ \code(\cU ,\mathbb R^{d})=\widehat{M(\C)}$. 
\end{proposition}
\begin{figure}[h] 
\centering
\definecolor{qqccqq}{rgb}{0,0.7,0}
\definecolor{qqqqzz}{rgb}{0,0,0.6}
\definecolor{ffqqqq}{rgb}{1,0,0}
\definecolor{ttfftt}{rgb}{0.2,1,0.2}
\definecolor{qqqqff}{rgb}{0,0,1}
\definecolor{ffqqtt}{rgb}{1,0,0.2}
\definecolor{uququq}{rgb}{0.25,0.25,0.25}
\definecolor{zzttqq}{rgb}{0.6,0.2,0}
\definecolor{qqzzzz}{rgb}{0,0.6,0.6}
\begin{center}
\begin{tikzpicture}[line cap=round,line join=round,>=triangle 45,x=1.4cm,y=1.4cm]

\draw [dash pattern = on 1pt off 1pt] (0,0) circle (1.5);
\fill [fill=black,fill opacity=0.3] (0,1.5) circle (0.1);
\fill [fill=black,fill opacity=0.3] ({1.5*cos(-30)},{1.5*sin(-30)}) circle (0.1);
\fill [fill=black,fill opacity=0.3] ({-1.5*cos(-30)},{1.5*sin(-30)}) circle (0.1);

 \draw [qqqqff,dash pattern=on 3pt off 3pt, fill=qqqqff, domain=-17.5:197.5, fill opacity = 0.25] plot ({2.5*cos(\x)}, {2.5*sin(\x)});
\draw [qqqqff, dash pattern=on 3pt off 3pt] ({2.5*cos(-17.5)}, {2.5*sin(-17.5)}) -- ({2.5*cos(197.5)}, {2.5*sin(197.5)});
 \draw [ffqqqq,dash pattern=on 3pt off 3pt, fill=ffqqqq, domain=102.5:317.5, fill opacity = 0.25] plot ({2.5*cos(\x)}, {2.5*sin(\x)});
\draw [ffqqqq, dash pattern=on 3pt off 3pt] ({2.5*cos(102.5)}, {2.5*sin(102.5)}) -- ({2.5*cos(317.5)}, {2.5*sin(317.5)});
 \draw [ttfftt,dash pattern=on 3pt off 3pt, fill=ttfftt, domain=222.5:437.5, fill opacity = 0.25] plot ({2.5*cos(\x)}, {2.5*sin(\x)});
\draw [ttfftt, dash pattern=on 3pt off 3pt] ({2.5*cos(222.5)}, {2.5*sin(222.5)}) -- ({2.5*cos(437.5)}, {2.5*sin(437.5)});
\begin{scriptsize}
\draw[color=black] (-0.15,1.65) node {$1$};
\draw[color=black] (-1.45,-0.6) node {$2$};
\draw[color=black] (1.15,-0.6) node {$3$};
\end{scriptsize}
\begin{large}
\draw[color=black] (0,2.2) node {$H_{\{1\}}$};
\draw[color=black] (-1.9053,-1.1) node {$H_{\{2\}}$};
\draw[color=black] (1.9053,-1.1) node {$H_{\{3\}}$};
\draw[color=black] (-1.299,0.75) node {$H_{\{1,2\}}$};
\draw[color=black] (1.299,0.75) node {$H_{\{1,3\}}$};
\draw[color=black] (0,-1.5) node {$H_{\{2,3\}}$};
\draw[color=black] (0,0) node {$H_{\{1,2,3\}}$};
\end{large}
\end{tikzpicture}
\end{center}
\caption{The oriented hyperplane arrangement $\{P_a\}$ and its chambers $H_\rho$. }
\label{fig-cartoon}
\end{figure}
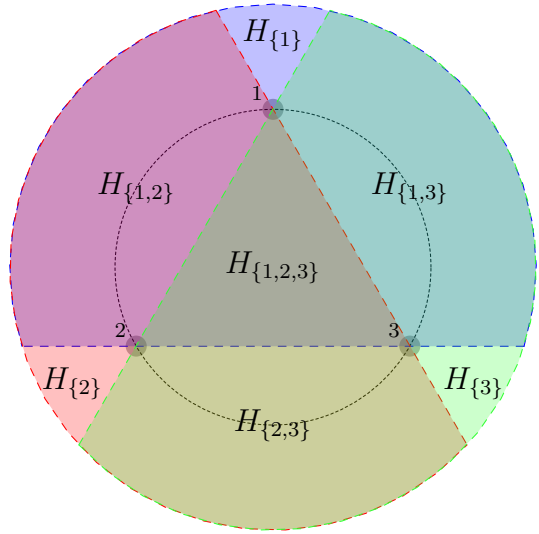

\begin{proof} Denote the maximal codewords as $M(\C) = \{\sigma_1, \sigma_2, \dots \sigma_k\}$. 
If $k=1$ this statement is trivially true. Assume $k\geq 2$ and consider a regular geometric $(k-1)$-simplex $\Delta^{k-1}$ in $\R^{k-1}$ constructed by evenly spacing vertices $[k]$ on the unit sphere  $S^{k-2} \subseteq \R^{k-1}$. Construct  a collection of  hyperplanes $\{P_a\}_{a=1}^k$ in $\R^{k-1}$ by taking $P_a$ to be the plane through the facet of $\partial \Delta^{k-1}$ which does not contain vertex $a$. Denote by  $H_a^+$  the {\it closed} half-space containing the vertex $a$   bounded by $P_a$ and by $H_a^-$   the complementary {\it open} half-space. Observe that this arrangement splits $\R^{k-1}$ into $2^{k}-1$ disjoint, non-empty, convex chambers
$$H_\rho =  \bigcap_{a \in \rho} H_a^+ \cap \bigcap_{b \not\in\rho}H_b^-,$$
indexed by all non-empty\footnote{The empty set is not included because under this definition, $H_\varnothing = \varnothing$.} subsets $\rho \subseteq {[k]} $. \\

For every $i\in [n]$ consider $\rho(i) \od \{a\in [k] \,\vert\, \sigma_a\ni i \} \subset[k]$, 
i.e.\   the collection of indices of the maximal codewords  $\sigma_a$ that contain $i$,
 and construct a collection of convex open sets $\U = \{U_i\}$ 
\begin{equation*}\label{eq:defUi}
U_i \od \coprod_{\rho\subseteq\rho(i)} H_\rho.
\end{equation*}
To show that the sets $U_i$ are convex and open, observe that the above  construction implies that we have the disjoint unions
$$\mathbb R^{k-1}=\coprod_{\rho\not=\varnothing}\, H_\rho \quad \text{ and } \quad H_b^+= \coprod_{\rho\ni b }\, H_\rho,$$ 
thus 
$$ \mathbb R^{k-1}\setminus U_i=
\left( \coprod_{\rho\neq\varnothing} H_\rho \right) \setminus \left(\coprod_{\rho\subseteq\rho(i)} H_\rho\right)=
\coprod_{\rho\not\subseteq\rho(i)} H_\rho =
\bigcup_{b\not\in \rho(i)}\left( \coprod_{\rho\ni b} H_\rho \right)=
\bigcup_{b\not\in \rho(i)}H_b^+ .
$$ 
Therefore, by de Morgan's Law,
\begin{equation}\label{eq:Ui} U_i= \mathbb R^{k-1}\setminus \left(\bigcup_{b\not\in \rho(i)}H_b^+ \right)=\bigcap_{b\not\in \rho(i)} H_b^-  .
\end{equation} 
This is an intersection of open convex sets, and therefore open and convex. Note that if $\rho(i) = [k]$, this is an intersection over an empty index, and we interpret this set as all of $\R^{k-1}$.\\

To show that $\code(\mathcal U,\mathbb R^{k-1})=  \widehat{M(\cC)}$, observe that because 
the chambers of the hyperplane arrangement satisfy 
 $H_\rho\cap H_{\nu}\neq \varnothing $ iff $\rho = \nu$, the atoms of the cover $\{U_i\}$ take the form 

\begin{align*} A^{\U}_\sigma=\bigcap_{i\in \sigma} U_i\setminus \left(\bigcup_{j\not\in \sigma}U_j \right)=
\left( 
\bigcup_{\rho\in \cap_{i\in \sigma}R_i}H_\rho
\right) \setminus 
\left(\bigcup _{\nu \in \cup_{j\not\in \sigma} R_j }H_\nu \right), 
\end{align*} 
 where each $R_i \od \{\rho \subseteq \rho(i) \}\subseteq 2^{[k]}\setminus \varnothing $ is the collection of the non-empty subsets of $\rho(i)$, 
and therefore 
\begin{equation*} 
 \code(\{U_i\} ,\mathbb R^{k-1})=\code(\{R_i\},2^{[k]}\setminus \varnothing ).
\end{equation*}
Now, observe that 
$$ \rho\in \bigcap_{i\in \sigma}R_i \iff \forall i \in \sigma,\, \rho\subseteq \rho(i) 
\iff \forall i \in \sigma,\, \forall a\in \rho, \, i\in \sigma_a
\iff \sigma\subseteq \bigcap _{a\in \rho} \sigma_a,
$$
and also that, 
$$ \rho\not \in \bigcup_{j\not\in \sigma} R_j \iff \forall j\not \in \sigma, \rho \not \subseteq \rho(i) \iff \forall j \not \in \sigma, \exists\, a\in \rho\text{ such that }j \not \in \sigma_a \iff \sigma \supseteq \bigcap_{a\in \rho} \sigma_a
$$
Therefore, 
$ \rho\in \bigcap_{i\in \sigma}R_i \setminus \left(\bigcup_{j\not\in \sigma} R_j \right)$ if and only if $ \sigma=\bigcap _{a\in \rho} \sigma_a$ and thus 
$$ \widehat{M(\cC)} =\code(\{R_i\},2^{[k]}\setminus \varnothing )= \code(\{U_i\} ,\mathbb R^{k-1}).$$

Lastly, we show that the cover $\U$ is non-degenerate. 
 By construction,  the half-spaces $H_a^-$ are open, convex and in general position. Thus  Lemma \ref{lemma:general:postiotion} guarantees that the cover $\mathcal{H}=\{H_a^-\}$ is non-degenerate and using Lemma \ref{lemma55} in the Appendix we conclude that for any non-empty $ \tau\subseteq [k]$, 
$\bigcap_{a\in \tau} \Cl(H_a^-) = \Cl\left(\bigcap_{a\in \tau} H_a^- \right)$. For any non-empty subset $\sigma\subseteq [n]$ we can   combine this with the equality \eqref{eq:Ui}  to obtain  
\begin{equation*}\label{eq:closurecommuteswithintersections} 
\Cl(\bigcap_{i\in \sigma}U_i) = 
\Cl(\bigcap_{i\in \sigma} \bigcap_{a\not \in \rho(i)} H_a^-) = 
\bigcap_{i\in \sigma} \bigcap_{a\not\in \rho(i)} \Cl(H_a^-) = 
\bigcap_{i\in \sigma} \Cl(\bigcap_{a\not\in \rho(i)} H_a^-) = 
\bigcap_{i\in \sigma} \Cl(U_i).
\end{equation*} 
Since $U_i$ are open we obtain 
\begin{align*}
\bigcap_{i\in \sigma} \partial U_i & = \bigcap_{i\in \sigma} \left( \Cl(U_i) \setminus U_i\right) 
 \subseteq \bigcap_{i\in \sigma} \left( \Cl(U_i) \setminus \bigcap_{i\in \sigma} U_i\right)
 = \left(\bigcap_{i\in \sigma} \Cl(U_i)\right) \setminus \bigcap_{i\in\sigma} U_i\\
 &= \Cl(\bigcap_{i\in\sigma} U_i) \setminus \bigcap_{i\in \sigma} U_i
 = \partial(\bigcap_{i\in \sigma} U_i).
\end{align*}
Therefore by Lemma \ref{lemma:nondeg}    
the  open and convex cover $\U$ is also non-degenerate. 
\end{proof}

\noindent As a corollary we obtain the main result of this section: \\

\noindent {\bf Theorem~\ref{t:mic}.} 
 Suppose $\C \subset 2^{[n]}$ is a max intersection-complete code. Then $\C$ is both open convex and closed convex with the  embedding dimension $d\leq \max\{2,(k-1)\}$, where $k$ is the number of facets of the complex $\Delta(\C)$. 

\begin{proof}
Note that the case of $k=1$, i.e.\   $M(\C)=\{[n]\}$, was  proved in \cite{Carina2015}. We first consider the case when the number of maximal codewords is $k \geq 3$ and begin by constructing convex regions $\{H_\rho\}_{\rho \in 2^{[k]}\setminus\varnothing}$ and the open convex cover $\{U_i\}_{i=1}^n$ as in the proof of Proposition  \ref{T:max_code} (see Figure \ref{fig-cartoon}). In this cover,  every atom that corresponds to a maximal codeword is unbounded, therefore  we can apply Lemma 3.2 using the open ball of radius 2 centered at the origin. This yields an open convex and non-degenerate cover, thus by  Theorem \ref{thm-non-degenerate} the code $\C$ is both open convex and closed convex.\\

If $1\leq k < 3$, we formally append $3-k$ empty maximal codewords $\{\gamma_j\}_{j=1}^{3-k}$ to $M(\C)$ and apply the same construction. Because the $\gamma_i$ are empty, they serve only to ``lift'' the construction to $\R^{2}$. The sets $U_i$ are contained entirely in $\bigcap_{j=1}^{3-k}H_{\gamma_j}^-$, but the $\gamma_i$ have no other effect on their composition. This allows us to carry out the rest of the above proof in the same way as in the case of $k\geq 3$. 
\end{proof}

 \bigskip 
\section{Appendix: supporting proofs}\label{the:appendix} 
\subsection{Proof of Lemma \ref{lemma:example:nonconvex}} 
\label{subsec5.1}
 \begin{proof} Consider the code $\C$ in   \eqref{eq:convexcounterexample} and assume that there exists a closed convex cover $\mathcal U =\{U_i\}$ in $\R^d$, with $\code(\mathcal U,\mathbb R^d )=\C$.
   Without loss of generality, we can assume that the $U_i$ are compact\footnote{If $U_i$ are not compact, then one can intersect them with a closed ball of large enough radius to obtain the same code.}.   Because $U_i$ are compact and convex one can pick points $x_{123}$, $x_{345}$, and $x_{156}$ in the closed atoms $A^{\U}_{123}$, $A^{\U}_{345}$ and $A^{\U}_{156}$ respectively so that for every  $a \in A^{\U}_{123}$, its distance to the closed line segment  $M = \overline{x_{345}x_{156}}$ satisfies\footnote{Because, $U_5$ is convex and contains the endpoints of $M$, $x_{123}\not\in M$. Moreover, since both $M$ and $A^{\U}_{123}$ are compact, the function $f(a)=\operatorname{dist}(a,M)$ achieves its minimum on $A^{\U}_{123}$.}
 $ \operatorname{dist}(a,M) \geq \operatorname{dist}( x_{123}, M) \ne 0$, i.e.\   $x_{123}$ minimizes the distance to the line segment $M$. 
Moreover, the points $x_{123},x_{156},x_{345}$ cannot be collinear. For the rest of this proof we will consider only the  convex hull of these three points (Figure \ref{fig-cartoon-proof-ofopen-but-not-closed}).

\newcommand{\n}{100}
\begin{figure}[h] 
\centering
\begin{tikzpicture}
\fill[black] (\n pt,0pt) circle (3pt) ; 
\fill[black] (-\n pt,0pt) circle (3pt) ; 
\fill[black] (0pt,\n pt) circle (3pt) ; 
\draw (-\n pt,0) -- (0,\n pt) -- (\n pt, 0);
\node [label={[xshift=\n+20 pt, yshift = -20pt]\textbf{$x_{156}$}}]{};
\node [label={[xshift=-20-\n pt, yshift = -20pt]\textbf{$x_{345}$}}]{};
\node [label={[xshift=0pt, yshift = \eval{\n } pt]\textbf{$x_{123}$}}]{};
\fill[black] (0.5*\n pt,0.5*\n pt) circle (3pt) ; 
\fill[black] (-0.3*\n pt,0.7*\n pt) circle (3pt) ; 
\draw (0.5*\n pt,0.5*\n pt)-- (-0.3*\n pt,0.7*\n pt);
\node [label={[xshift=0.5*\n+20 pt, yshift = 0.5*\n -16 pt]\textbf{$x_{126}$}}]{};
\node [label={[xshift=-0.3*\n-21 pt, yshift = 0.7*\n - 15 pt]\textbf{$x_{234}$}}]{};
\draw[line width=0.8mm] (\n pt, 0) -- (-\n pt,0);
\node [label={[xshift=0 pt, yshift = -25 pt]\textbf{M}}]{};
\fill[black] (0.1*\n pt,0.6*\n pt) circle (3pt) ; 
\node [label={[xshift=15 pt, yshift = 0.5*\n-16 pt]\textbf{$y_{123}$}}]{};
\end{tikzpicture}
 \caption{}
 \label{fig-cartoon-proof-ofopen-but-not-closed}
\end{figure}
Consider  the closed line segment $L = \overline{x_{123}x_{156}}$. Because $U_1$ is convex, $L \subset U_1$, therefore the  code \eqref{eq:convexcounterexample} of the cover imposes that 
$$L\subset A^{\U}_{123} \sqcup A^{\U}_{12} \sqcup A^{\U}_{126} \sqcup  A^{\U}_{16} \sqcup A^{\U}_{156}.$$ 
Because each of the atoms above is contained  in either $ U_2$ or $U_6$,  $L\subset U_2 \cup U_6$. Since $L$ is connected and the sets $U_2\cap L$ and $U_6\cap L$ are closed  and non-empty,  we conclude  that  $U_2 \cap U_6 \cap L\subset A^{\U}_{126}$ must be non-empty, thus there exists a point $x_{126} \in A^{\U}_{126} \cap L$ that lies  in the interior of $L$. By the same argument, there also exist points
\begin{align*}
x_{234} \in A^{\U}_{234} &\text{ in the interior of }\overline{x_{123}x_{345}} \subset U_3, \text{ covered by }U_2\text{ and }U_4.\\
y_{123}\in A^{\U}_{123} &\text{  in the interior of } \overline{x_{234}x_{126}} \subset U_2,\text{ covered by }U_1\text{ and }U_3.
\end{align*}
and these points must lie on the interiors of their respective line segments (Figure \ref{fig-cartoon-proof-ofopen-but-not-closed}).\\

Finally we observe that  because the point $y_{123}\in A^{\U}_{123}$  lies in the interior of a line segment  $ \overline{x_{234}x_{126}}$, it also lies  in the interior of the closed triangle $\triangle(x_{123},x_{156},x_{345})$,  and thus  $d(y_{123},M) < d(x_{123},M)$. This yields a contradiction, since we chose $x_{123}\in A^{\U}_{123}$ to have the minimal distance to the line segment $M$.
 \end{proof}
 
 \medskip 
 \subsection{Proofs of lemmas, related to the non-degeneracy condition.}
 \label{sec:nondeg:condition:lemmas}
We shall need the following several   lemmas.  The following lemma is well-known (see e.g. \cite{HatcherNotes}, exercises in Chapter 1), nevertheless we give its proof for the sake of completeness.

\begin{lemma}\label{Lemma4.1} For any finite  cover $\mathcal{U} = \{U_i\}_{i=1}^n$ and a  subset $\sigma\subseteq  [n]$, the following  hold:
\begin{align} 
\label{eq:4l}\Cl(\bigcup_{i\in \sigma} U_i) &= \bigcup_{i\in \sigma} \Cl(U_i),\\
 \label{eq:3half} \Cl(\bigcap_{i\in \sigma} U_i) &\subseteq \bigcap_{i\in \sigma} \Cl(U_i), \\
\label{eq:1l}\Int(\bigcap_{i\in \sigma} U_i) &= \bigcap_{i\in \sigma} \Int(U_i), \\
\label{eq:1half}\Int(\bigcup_{i\in \sigma} U_i) &\supseteq \bigcup_{i\in \sigma} \Int(U_i). 
\end{align} 
\end{lemma}

\begin{proof} 
Observe that since $U_i\subseteq \Cl(U_i)$, we have 
$\bigcup_{i\in \sigma}  U_i \subseteq \bigcup_{i\in \sigma} \Cl(U_i)$ and thus 

\begin{equation} \label{eq:4half}
\Cl(\bigcup_{i\in\sigma} U_i) \subseteq \Cl\left(\bigcup_{i\in \sigma} \Cl(U_i)\right) = \bigcup_{i\in \sigma} \Cl(U_i).
\end{equation}
Similarly,  we find the inclusion \eqref{eq:3half}. Using $U_i\supseteq \Int(U_i)$, one also obtains the inclusion \eqref{eq:1half} and the inclusion 
\begin{align} 
 \label{eq:inclusion:int:cap}
 \Int(\bigcap_{i\in \sigma} U_i) &\supseteq  \bigcap_{i\in \sigma} \Int(U_i).
\end{align} 
Observe  that for any  $j\in \sigma$, 
$\Cl(U_j) \subseteq \Cl(\bigcup_{i\in \sigma} U_i) $  and $\Int(U_j) \supseteq \Int(\bigcap_{i\in \sigma} U_i) $, 
thus we obtain  $\bigcup_{i\in \sigma} \Cl(U_i) \subseteq \Cl(\bigcup_{i\in \sigma} U_i)$ and $\bigcap_{i\in \sigma} \Int(U_i) \supseteq \Int(\bigcap_{i\in \sigma} U_i)$. 
These combined with  \eqref{eq:4half} and \eqref{eq:inclusion:int:cap} yield    \eqref{eq:4l} and  \eqref{eq:1l} respectively.  
\end{proof}

\medskip

\begin{lemma}\label{L:open_nondegen}
  Suppose $\mathcal{U}=\{U_i\}_{i\in [n]}$ is an open and convex cover such that  for every non-empty subset $\tau \subseteq [n]$, $\bigcap_{i \in \tau}\partial U_i \subseteq \partial(\bigcap_{i\in \tau}U_i)$.
  Then every atom of $\mathcal{U}$ is top-dimensional.
  \end{lemma}

  \begin{proof}
    Assume the converse, i.e.  there exists non-empty $ \sigma\subset [n]$ and an open subset $B \subseteq \mathbb{R}^d$  such that that $A_\sigma^{\mathcal{U}} \cap B \neq \varnothing$  and $\text{int}(A_\sigma^\mathcal{U} \cap B) = \varnothing$.
    Let $x \in A_\sigma^\mathcal{U} \cap B$, and   denote by $\tau\subset [n]\setminus \sigma$, the maximal subset such that $x\in \cap_{j\in \tau} \partial U_j$. Note that $\tau$ is non-empty\footnote{If $x\notin \partial U_j$  $\forall j\notin \sigma$, then (because $U_i$ are open) there exists a small open ball $B^\prime \ni x $ such that 
     $B^\prime\subset   A_\sigma^\mathcal{U} $, thus $\text{int}(A_\sigma^\mathcal{U} \cap B) \supseteq \text{int}(A_\sigma^\mathcal{U} \cap B\cap B^\prime)\neq \varnothing$, a contradiction.} and therefore (using the assumption of the lemma)  $x\in \partial\left(  \cap_{j\in \tau} U_j\right)$.  
     Denote by $\varepsilon_0>0$ the maximal radius  such that  the open ball $B_{\varepsilon_0}(x)$ satisfies (a) $B_{\varepsilon_0}(x) \subset B\cap \cap_{i\in \sigma}U_i$ and  (b) for every $l\not\in  \left( \sigma\cup \tau\right)$,  $B_{\varepsilon_0}(x)\cap U_l=\varnothing$.

     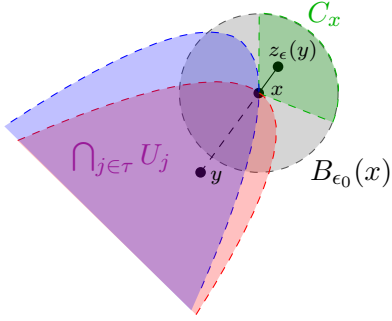
\begin{figure}[h] 
\definecolor{qqccqq}{rgb}{0,0.7,0}
\definecolor{qqqqzz}{rgb}{0,0,0.6}
\definecolor{ffqqqq}{rgb}{1,0,0}
\definecolor{ttfftt}{rgb}{0.2,1,0.2}
\definecolor{qqqqff}{rgb}{0,0,1}
\definecolor{ffqqtt}{rgb}{1,0,0.2}
\definecolor{uququq}{rgb}{0.35,0.35,0.35}
\definecolor{uzuzuz}{rgb}{0.05,0.05,0.05}
\definecolor{zzttqq}{rgb}{0.6,0.2,0}
\definecolor{qqzzzz}{rgb}{0,0.6,0.6}
\definecolor{kk}{rgb}{0,0,0}
\definecolor{eeaaee}{rgb}{0.6, 0.1, 0.6}
\begin{center}
\begin{tikzpicture}
\draw[uququq, dash pattern = on 3pt off 3 pt] (30pt, 30pt) circle (30pt);
\fill[uququq, fill opacity = 0.25] (30pt, 30pt) circle (30pt);
\fill[black] (8pt,0pt) circle (2pt) ;
\fill[black] (30pt,30pt) circle (2pt) ;
\fill[black] (37.3pt,40pt) circle (2pt) ;
\draw[black, dash pattern = on 3 pt off 3 pt] (8pt, 0pt) -- (30pt, 30pt);
\draw[black] (30pt, 30pt) -- (37.3pt, 40pt);
\fill[qqccqq, fill opacity = 0.25] (30pt,30pt) -- (58pt, 19pt) arc (-22:90:30pt) -- cycle;
\draw[qqccqq, dash pattern = on 3 pt off 3 pt] (30pt,30pt) -- (58pt, 19pt) arc (-22:90:30pt) -- cycle;
\draw [ffqqqq, dash pattern = on 3pt off 3pt, domain=-1.6:1.7, rotate around={130:(0pt, 30pt)}]  plot ({\x-0.8}, {(\x)^2+0.22});
\draw [qqqqff, domain=-1.89:1.5, dash pattern = on 3 pt off 3 pt, rotate around={140:(0,30pt)}] plot ({\x}, {(\x+0.2)^2});
\fill [ffqqqq, fill opacity = 0.25, domain=-1.6:1.7, rotate around={130:(0pt, 30pt)}]  plot ({\x-0.8}, {(\x)^2+0.22});
\fill [qqqqff, domain=-1.98:1.5, fill opacity = 0.25, rotate around={140:(0,30pt)}] plot ({\x}, {(\x+0.2)^2});
\begin{scriptsize}
\end{scriptsize}

\begin{scriptsize}
\draw[color=black] (14pt, -2pt) node {$y$};
\draw[color=black] (37pt, 31pt) node {$x$};
\draw[color=black] (43pt, 46pt) node {$z_\epsilon(y)$};
\end{scriptsize}
\draw[color=black] (65pt, 0pt) node {$B_{\epsilon_0}(x)$};
\draw[color=qqccqq] (55pt, 60pt) node {$C_x$};
\draw[color=eeaaee] (-22pt, 5pt) node {$\bigcap_{j \in \tau}U_j$};

\end{tikzpicture}
\end{center}
\caption{
Construction of points in $\text{int}(A^{\mathcal{U}}_\sigma \cap B)$ from the proof of Lemma \ref{L:open_nondegen}.}
 \label{fig-open-nondegen}
\end{figure}
     
     Observe that for every  point $y\in    \cap_{j\in \tau} U_j $ and   every $\varepsilon\in (0,\varepsilon_0)$,  the point  $z_\epsilon(y)=x+\varepsilon \frac{x-y}{\| x-y\|}$ satisfies $z_\epsilon(y)\not\in U_j$  for every $j\not\in \sigma$. This is because for every $j\in \tau$, the open set $U_j$ is convex, thus if  $x\in \partial U_j$, and $y\in U_j$, then $z_\epsilon(y)\not\in U_j$, as in Figure \ref{fig-open-nondegen}.  We thus conclude that $z_\epsilon(y)\in A_\sigma^\mathcal{U} \cap B$.  Since the intersection $ \cap_{j\in \tau} U_j $ is open, the totality of all such  points $z_\epsilon(y)$ form an open cone $C_x\subseteq A_\sigma^\mathcal{U} \cap B$. Therefore $\text{int}(A_\sigma^\mathcal{U} \cap B)  \supseteq \text{int}(C_x)\neq \varnothing$, a contradiction. 
 \end{proof}

\begin{lemma}\label{lemma:general:postiotion}
Suppose $\mathcal{U}$ is an open and convex cover in general position, then $\mathcal{U}$ is a non-degenerate cover.
\end{lemma}
\begin{proof}
Assume $\mathcal{U}$ is in general position, open, convex, yet {\it not}  non-degenerate. Then, by Lemma \ref{L:open_nondegen} there exists  a non-empty subset $\sigma \subseteq [n]$ so that $\bigcap_{i \in \sigma}\partial U_i \not\subseteq \partial(\bigcap_{i\in \sigma}U_i)$. Let's choose $x \in \left( \bigcap_{i \in \sigma}\partial U_i \right) \setminus \left( \partial \bigcap_{i \in \sigma} U_i\right) $. Suppose there exists $z \in \bigcap_{i \in \sigma}U_i$, then  the open line segment between $x$ and $z$ is contained in $\bigcap_{i\in \sigma}U_i$, and thus  $x \in \partial (\bigcap_{i \in \sigma}U_i)$, a contradiction. Therefore, $\bigcap_{i \in \sigma}U_i= \varnothing$, and  for every $\tau\supseteq \sigma$, $\tau \not \in \text{code}(\mathcal{U},\mathbb{R}^d)$.

For any  $\varepsilon > 0$, define an open cover $\mathcal{V}(\varepsilon)=\{{V_i}(\varepsilon)\}$ by $V_i(\varepsilon) = U_i \cup B_\varepsilon(x)$ for $i \in \sigma$ and $V_j(\varepsilon) = U_j$ otherwise. Notice that $\bigcap_{i \in \sigma} V_i(\varepsilon) = B_\varepsilon(x)$.  Thus  for any $ \varepsilon > 0$, there exists $  \tau \supseteq \sigma$ with $\tau \in \text{code}(\mathcal{V}(\varepsilon), \mathbb{R^d})$. Because $x$ lies in the boundary of $U_i$ for each $i \in \sigma$, each $V_i(\varepsilon)$ is no more than   $\varepsilon$ away from $U_i$ w.r.t. the Hausdorff distance.  Therefore $\U$ is not in general position, a contradiction. 
\end{proof}

\begin{lemma} \label{lemma:key:closed:lemma} Assume that every atom of the cover  $\U=\{U_i\}$ is top-dimensional, 
 i.e.\   any non-empty intersection
 with an open set $B\subseteq \mathbb R^d $ has non-empty interior, 
 and the subsets $U_i$ are closed and convex, 
then for any non-empty $\tau\subseteq [n]$, 
\begin{align}\label{eq:key:closed:lemma}
\bigcap_{i\in \tau}\partial U_i&\subseteq \partial \left(\bigcup_{i\in \tau} U_i  \right), \\
\label{eq:condition:i}
\bigcap_{i\in \tau}\partial U_i&\subseteq \partial \left(\bigcap_{i\in \tau} U_i  \right).
\end{align}
\end{lemma} 
\begin{proof} To show \eqref{eq:key:closed:lemma} assume the converse. Then  there exist  a point $x\in \left(\bigcap_{i\in \tau}\partial U_i \right) \bigcap \Int\left(\bigcup_{i\in \tau} U_i  \right)$ at the interior,   and  an open ball $B\ni x$,  such that $B \subseteq \bigcup_{i\in\tau}U_i$.
First, let us show that  these assumptions imply that 
\begin{equation} \label{eq:false} B\cap \bigcap_{i \in \tau}\Int(U_i) =\varnothing.
\end{equation}
Indeed, if there existed  a point $y \in B\cap \bigcap_{i \in \tau}\Int(U_i)$, then for every $\varepsilon>0$ such that  $z=x+\varepsilon(x-y)\in B$ and every $i \in \tau$, $z \not \in U_i$ by convexity of  $U_i$\footnote{This is because $y\in \Int(U_i)$, $x\in \partial U_i$ and $U_i$ is convex, thus for every $\varepsilon>0$,  $z=x+\varepsilon(x-y)\not\in U_i$.}. 
This implies $B  \not\subseteq\bigcup_{i \in \tau}U_i$, a contradiction, thus \eqref{eq:false} holds. \\

Denote by  $\rho \supseteq \tau$ the element of $\code\left( \{U_i\},\mathbb R^d\right) $  such that  $x \in A^{\U}_\rho = \bigcap_{i \in \rho}U_i \setminus \bigcup_{j \not \in \rho} U_j$.  Because the sets $U_j$ are closed, we can choose the open ball $B\ni x$, that satisfies \eqref{eq:false} so that it is  disjoint from $\bigcup_{j \not\in \rho}U_j$. Therefore, using \eqref{eq:1l}, we obtain 

$$\Int(B \cap A^{\U}_\rho ) = \Int(B \cap \bigcap_{i \in \rho} U_i ) \subseteq \Int(B \cap \bigcap_{i \in \tau} U_i  ) = B \cap  \bigcap_{i \in \tau}\Int(U_i)= \varnothing.$$ 
Since $x\in B \cap A^{\U}_\rho$, this contradicts the non-degeneracy of $\U$, and thus finishes the proof of \eqref{eq:key:closed:lemma}.  \\

To prove \eqref{eq:condition:i}, consider  $x\in \bigcap_{i\in \tau}\partial U_i\subseteq \bigcap_{i\in \tau} U_i   $.   
Because of  \eqref{eq:key:closed:lemma},  any open neighborhood $O\ni x$ satisfies $O\not\subseteq \bigcup_{i\in \tau} U_i $ and thus  $O\not\subseteq \bigcap_{i\in \tau} U_i $. Therefore  $x\in \partial \left(\bigcap_{i\in \tau} U_i  \right)$. 
\end{proof} 
Note that if the condition that the sets $U_i$ are convex  is violated, then the conclusions of the above lemma may not hold. 
For example, the  sets  $U_1=\{ (x,y)\in \mathbb R^2 \,\vert\,   y\leq x^2\}$ and  $U_2=\{ (x,y)\in \mathbb R^2 \,\vert\,   y\geq -x^2\}$ do not satisfy the inclusion \eqref{eq:key:closed:lemma}.

\begin{lemma}\label{lemma55}If the cover  $\mathcal{U}=\{U_i\}_{i\in [n]}$  is non-degenerate, then for every non-empty subset $\sigma\subseteq [n]$ 
\begin{align} 
\label{eq:2l}
U_i  \text{ are closed and convex} &\implies \Int(\bigcup_{i\in \sigma} U_i) = \bigcup_{i\in \sigma} \Int(U_i),
\\
\label{eq:3l} 
U_i  \text{ are open  and convex} &\implies  \Cl(\bigcap_{i\in \sigma} U_i) = \bigcap_{i\in \sigma} \Cl(U_i).
\end{align}
\end{lemma}
\begin{proof}
First, we show that 
 if the cover $\mathcal U$ is non-degenerate and closed convex, then 
\begin{equation} \label{eq:interiorinclusions}  \Int(\bigcup_{i\in \sigma} U_i) \subseteq \bigcup_{i\in \sigma} \Int(U_i).\end{equation}
It suffices to show that if   $x\notin \bigcup_{i\in \sigma} \Int(U_i)$, then  $x\in \partial(\bigcup_{i\in \sigma} U_i)\bigcup \left( \mathbb R^d \setminus \bigcup_{i\in \sigma} U_i \right) $. 
 If $x\notin \bigcup_{i\in \sigma} U_i$, then this is true, thus we can assume that the 
 set $\tau\od \{i\in \sigma \, \vert \, x\in U_i\} $  is non-empty, 
 and since $x\notin \bigcup_{i\in \sigma} \Int(U_i)$, we conclude that $x\in \bigcap_{i\in \tau}\partial U_i$. Thus by Lemma \ref {lemma:key:closed:lemma} (\eqref{eq:key:closed:lemma}),    $x\in \partial(\bigcup_{i\in \tau } U_i)$. Now observe that  $ \bigcup_{i\in \sigma} U_i= A\cup B$ with $A\od \bigcup_{i\in \tau } U_i$ and $B\od \bigcup_{j\in \sigma\setminus \tau} U_j$.  
 Since $x\notin B$, and $B$ is closed, there exists an open neighborhood $O\ni x $ with $O\cap B=\varnothing $. Therefore, using \eqref{eq:1l} we obtain   that 
 $$O\cap  \Int (A)    =  \Int (O\cap  A) =  \Int (O\cap ( A\cup B))= 
 O\cap \Int\left(A\cup B\right),$$
 and thus we  conclude  
 $$ x\in \partial A \cap  O= \left ( A\setminus  \Int A  \right)  \cap O= 
   \left (  \left (  A\cup B \right)\setminus \left(  \Int \left (  A\cup B\right) \cap O   \right) \right)  \cap O=\partial \left (  A\cup B\right) \cap O.
   $$
  \noindent Thus,  $x\in\partial\left(A\cup B \right)=\partial\left( \bigcup_{i\in \sigma} U_i\right)$,  which proves \eqref{eq:interiorinclusions}. Combined with \eqref{eq:1half} in Lemma \ref{Lemma4.1}, this finishes the proof of \eqref{eq:2l}.
 
 To prove   \eqref{eq:3l}, taking into account  \eqref{eq:3half}, we need to show that 
$ \Cl(\bigcap_{i\in \sigma} U_i) \supseteq \bigcap_{i\in \sigma} \Cl(U_i) $. 
Assume the converse, then there exists $x\in  \bigcap_{i\in \sigma} \Cl(U_i)  $  and $r>0$ such that 
\begin{equation} \label{eq:epsball}
 \forall \varepsilon\in (0,r) \text{ the open } \varepsilon\text{-ball }  B_\varepsilon(x) \text{ satisfies }  B_\varepsilon(x) \cap \bigcap_{i\in \sigma} U_i=\varnothing.  
\end{equation} 
Denote $\tau\od \{ i\in \sigma \, \vert \, x\in \partial U_i\}$; we can assume that  $\tau$  is non-empty (otherwise, $x\in  \Cl \left(\bigcap_{i\in \sigma} U_i \right)$). 
Using the condition (ii) of Definition \ref{dfn:nondegenerate} we conclude  $x\in \bigcap_{i\in \tau}\partial U_i\subseteq \partial \left(  \bigcap_{i\in \tau} U_i   \right) $, thus for  every  open $\varepsilon$-ball $B_\varepsilon(x)$ centered at $x$,  $B_\varepsilon(x)\cap  \bigcap_{i\in \tau} U_i\neq \varnothing$. Because $x\in \bigcap_{j\in \sigma\setminus \tau} U_j$, and $U_j$ are open,  for a sufficiently small  $\varepsilon$, $B_\varepsilon(x)\subset  \bigcap_{j\in \sigma\setminus \tau} U_j$. Thus $B_\varepsilon(x) \cap  \bigcap_{i\in \sigma} U_i\neq \varnothing$, which contradicts \eqref{eq:epsball}. This finishes the proof of \eqref{eq:3l}.
\end{proof}

%

 \medskip 
 \subsection{Proof of Theorem \ref{thm-non-degenerate}.}
 \label{sec:nondeg:proofs}
 
\begin{proof}
We need to show   that if $\mathcal{U}$ is convex and non-degenerate, then the cover of closures $\Cl(\mathcal{U}) \od \{\Cl(U_i)\}$ and the cover of interiors $\Int(\mathcal{U}) \od  \{\Int(U_i)\}$ have the same code as $\mathcal{U}$. 
First, we show that $\code(\mathcal{U}) = \code(\Cl(\mathcal{U}))$. Let $A^{\U}_\sigma$ denote an atom of $\U$ and $A^{\Cl(\U)}_\sigma$ denote  the corresponding atom of $\Cl(\U)$. If $A^{\U}_\sigma = \varnothing $, then using \eqref{eq:3l} and \eqref{eq:4l} we conclude that 
\begin{align*}
\bigcap_{i \in \sigma} U_i &\subseteq \bigcup_{j \notin \sigma} U_j
\implies \Cl(\bigcap_{i \in \sigma} U_i) \subseteq \Cl(\bigcup_{j \notin \sigma} U_j)
 \implies
\bigcap_{i \in \sigma} \Cl(U_i) \subseteq \bigcup_{j \notin \sigma} \Cl(U_j),
\end{align*}
and thus  $A^{\Cl(\U)}_\sigma= \varnothing $. Therefore, $\code(\Cl(\U)) \subseteq \code(\U)$. On the other hand,  using \eqref{eq:4l} we obtain 
 \begin{align*}
A^{\Cl(\U)}_\sigma &= \bigcap_{i\in \sigma} \Cl(U_i) \setminus \bigcup_{j\notin \sigma} \Cl(U_j)
= \bigcap_{i \in \sigma} \Cl(U_i) \setminus  \Cl(\bigcup_{j\notin \sigma} U_j) \\
&= \left(\bigcap_{i \in \sigma} \Cl(U_i) \setminus  \bigcup_{j\notin \sigma} U_j\right)\setminus \left(\Cl(\bigcup_{j\notin \sigma} U_j) \setminus \bigcup_{j\notin \sigma} U_j \right)
\supseteq A^{\U}_\sigma \setminus \partial\left(\bigcup_{j\notin \sigma} U_j\right).
\end{align*}
Thus, if $A^{\U}_\sigma$ is non-empty, since it is top-dimensional while $\partial\left(\bigcup_{i\notin \sigma} U_i\right)$ is of codimension one, $A^{\U}_\sigma \not\subseteq \partial\left(\bigcup_{j\notin \sigma} U_j\right)$, implying $A^{\Cl(\U)}_\sigma \ne \varnothing $, and  thus, $\code(\U) = \code(\Cl(\U))$. \\

Next, we show that $\code(\Int(\U)) = \code(\U)$. Let $A^{\U}_\sigma$ be an  atom of $\U$ and $A^{\Int(\U)}_\sigma$ be the corresponding atom of $\Int(\U)$. If $A^{\U}_\sigma = \varnothing $, then using \eqref{eq:1l} and \eqref{eq:2l} we conclude that 
\begin{equation*}
\bigcap_{i \in \sigma} U_i \subseteq \bigcup_{j \notin \sigma} U_j
\,\implies\,  \Int(\bigcap_{i \in \sigma} U_i)  \subseteq \Int(\bigcup_{j \notin \sigma} U_j)
\,\implies\,
\bigcap_{i \in \sigma} \Int(U_i) \subseteq \bigcup_{j \notin \sigma} \Int(U_j),
\end{equation*}
 which implies $A^{\Int(\U)}_\sigma= \varnothing $. Therefore, $\code(\Int(\U)) \subseteq  \code(\U)$. On the other hand, using \eqref{eq:1l}  we obtain 
\begin{align*}
A^{\Int(\U)}_\sigma = \bigcap_{i\in \sigma} \Int(U_i) \setminus \bigcup_{j\notin \sigma} \Int(U_j) 
\supset \Int(\bigcap_{i\in \sigma} U_i) \setminus \bigcup_{j\notin \sigma} U_j=\\
= \left(\bigcap_{i\in \sigma} U_i \setminus \bigcup_{j\notin \sigma} U_j\right) \setminus \partial\left(\bigcap_{i\in \sigma} U_i\right) 
= A^{\U}_\sigma \setminus\partial\left(\bigcap_{i\in \sigma} U_i\right).
\end{align*}
Thus, if $A^{\U}_\sigma$ is non-empty, since it is top-dimensional while  $\partial\left(\bigcap_{i\in \sigma} A_i\right)$ is of codimension one,   $A^{\Int(\U)}_\sigma\ne \varnothing $. Therefore, $\code(\U) = \code(\Int(\U))$. 
\end{proof}

 \bigskip 
\subsection{Proof of Lemma \ref{lem:Ball:rules}.}\label{apendix:monotone}
In order to prove Lemma \ref{lem:Ball:rules}  we will need the following two lemmas. 
\begin{lemma} \label{lemma:code:union} Let $\W = \{W_i\} $ be a collection of sets, $W_i\subseteq X$, and $\C= \code(\W,X)$. Assume that $Q$ is a proper subset of some atom of $\W$, i.e.\   $\varnothing \neq Q\subsetneq  A^{\W}_\alpha$,  for a  non-empty $\alpha\in \C$. Then for any $\sigma_0 \subsetneq \alpha$, the cover $\V=\{V_i\} $ by the sets 
\begin{equation}\label{eq:Vi} V_i=  \begin{cases}  W_i, & \mbox{if }  i\in \sigma_0,  \\ W_i\setminus Q  , & \mbox{if }  i\not \in \sigma_0\end{cases} \end{equation} 
adds the codeword $\sigma_0$ to the original code, i.e.\   $ \code(\V,X) = \code(\W,X)\cup \{\sigma_0\}$.
\end{lemma} 

\begin{proof} Since $Q\subsetneq  A^{\W}_\alpha$, $\code(\{V_i \cap (X \setminus Q)\}, X \setminus Q) = \code(\W, X)$.  Moreover,  because  $\sigma_0\subset \alpha$, $\code(\{V_i \cap Q\}, Q) = \{\sigma_0\}$ by construction. Finally, observe that if $X = Y \sqcup Z$, then $\code(\V, X) = \code(\{V_i \cap Y\}, Y) \cup \code(\{V_i \cap Z\}, Z)$, therefore we obtain 
\begin{equation*}\code(\V, X) = \code(\{V_i \cap (X \setminus Q)\}, X \setminus Q) \cup \code(\{V_i \cap Q\}, Q) = \code(\W,X) \cup \{\sigma_0\}.\end{equation*}
\end{proof}

Recall that  $M(\C)\subset \C $ denotes  the set of maximal codewords of $\C$.  A subset $A\cap B $ of a topological space is called {\it relatively open in}   $B$  if it is an open set in the induced topology of the subset $B$.

\begin{lemma} \label{lemma:chord-cutter}  
Let $\U = \{U_i\}$ be an open convex cover in $\mathbb R^d$, $d\geq 2$,  with  $\C=\code\left(\mathcal U,X\right)$. Assume that  there exists an 
open Euclidean ball $B\subset \mathbb R^d $ such that $\code(\{B\cap U_i\},B\cap X) = \C$, and for every maximal set  $\alpha \in M(\C)$, 
the set   $\partial B \cap  \Cl \left( \bigcap_{i\in \alpha }U_i \right)$  is non-empty and  is relatively open in  $\partial B$. 
 Then for any simplicial violator 
$\sigma_0\in \Delta(\C)\setminus \C $, there exists an open convex cover $\V =\{V_i\}$ with $V_i\subseteq U_i$, so that $ \code(\mathcal V,B\cap  X )=\C\cup \sigma_0$, and the cover $\V$ satisfies the  same condition above with  the same open ball  $B$. Moreover, if the cover $\U = \{U_i\}$ is non-degenerate, then the cover $\V$ can also be chosen to be non-degenerate. 
\end{lemma}

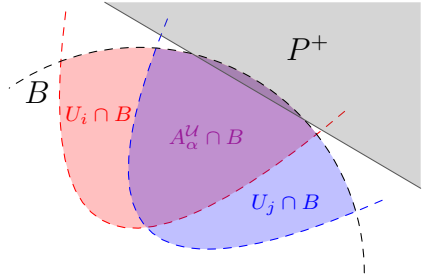
\begin{figure}[h] 
\definecolor{qqccqq}{rgb}{0,0.7,0}
\definecolor{qqqqzz}{rgb}{0,0,0.6}
\definecolor{ffqqqq}{rgb}{1,0,0}
\definecolor{ttfftt}{rgb}{0.2,1,0.2}
\definecolor{qqqqff}{rgb}{0,0,1}
\definecolor{ffqqtt}{rgb}{1,0,0.2}
\definecolor{uququq}{rgb}{0.35,0.35,0.35}
\definecolor{uzuzuz}{rgb}{0.05,0.05,0.05}
\definecolor{zzttqq}{rgb}{0.6,0.2,0}
\definecolor{qqzzzz}{rgb}{0,0.6,0.6}
\definecolor{kk}{rgb}{0,0,0}
\definecolor{eeaaee}{rgb}{0.6, 0.1, 0.6}
\begin{center}
\begin{tikzpicture}[line cap=round,line join=round,>=triangle 45,x=1.5cm,y=1.5cm]
\draw [ffqqqq, dash pattern = on 3pt off 3pt, domain=0:2.6, rotate around={-30:(1,2.5)}] plot ({\x}, {(\x-1.2)^2+0.5});
\fill [ffqqqq, fill opacity = 0.25, domain = 0.18:2.47, rotate around={-30:(1,2.5)}] plot ({\x}, {(\x-1.2)^2+0.5}) arc [radius = 2, start angle = 68, delta angle = 65] -- cycle;
\draw [qqqqff, domain=0.8:3.3, dash pattern = on 3 pt off 3 pt, rotate around={-45:(2,2)}] plot ({\x}, {(\x-2)^2});
\fill [qqqqff, fill opacity = 0.25, domain=0.908:3.19, rotate around={-45:(2,2)}] plot ({\x}, {(\x-2)^2}) arc [radius = 2, start angle = 61, delta angle = 68] -- cycle;
\draw [uququq] (3, 0.95) -- (0.2, 2.6);
\fill [uququq, fill opacity = 0.25] (3, 0.95) -- (0.2, 2.6) -- (3, 2.6) -- cycle;
\draw [kk, dash pattern = on 3 pt off 3 pt] (2.5,0.2) arc [radius=2, start angle=0, delta angle=125];
\begin{scriptsize}
\draw[color=ffqqqq] (0.145, 1.6) node {$U_i \cap B$};
\draw[color=qqqqff] (1.8, 0.8) node {$U_j \cap B$};
\draw[color=eeaaee] (1.1, 1.4) node {$A^{\mathcal{U}}_\alpha \cap B$};
\end{scriptsize}
\begin{large}
\draw[color=black] (2,2.2) node {$P^+$};
\draw[color=black] (-.4, 1.8) node {$B$};
\end{large}
\end{tikzpicture}
\end{center}
\caption{The oriented half space $P^+$ is chosen to intersect the ball $B$ inside $A^{\mathcal{U}}_\alpha$.}
 \label{fig-sphere-cord}
\end{figure}

\begin{proof} Choose a facet $\alpha\in  M(\C)$ such that   $\alpha \supsetneq  \sigma_0$. Because  $\alpha$ is a facet of  $\Delta(\C)$, the atom $A^{\mathcal U }_{\alpha }=\cap_{i\in \alpha} U_i$ is convex open and (by assumption) has a non-empty relatively open intersection with the Euclidean sphere $\partial B$. This implies that  we  can always select an  oriented and {\it closed}  half-space   $P^+\subset \mathbb R ^d $  such that   $P^+ \cap B \subset   A^{\mathcal U }_{\alpha}$,      and $\left( A^{\mathcal U }_{\alpha} \cap  B \right) \setminus   P^+  \neq \varnothing  $ has relatively open intersection with the sphere $\partial B$ (see Figure \ref{fig-sphere-cord}). \\

We define two open covers, $\W=\{W_i\}$, with $W_i\od U_i \cap B$ and  $\V=\{V_i\}$ via the equation \eqref{eq:Vi}, with $Q=  P^+\cap B$.  
 We thus can use Lemma \ref{lemma:code:union}, and conclude that 
 $\code(\V,  X\cap B   )= \code(\{B\cap U_i\}  , B \cap X ) \cup\sigma_0 =\C \cup \sigma_0$. 
 Note that by construction the sets $V_i$ are open and convex, 
 moreover, the cover $\V$ automatically satisfies the same condition on the atoms of facets of $\Delta(\C)$.\\
 
 Finally, if $\U$ is non-degenerate, then   $\V$  is also non-degenerate. Indeed,  because $A^\U_\alpha$ is open,  the only two atoms that were changed, $A^\V_{\alpha}=\left( A^\U_\alpha\cap B\right)  \setminus  P^+$ and $A^\V_{\sigma_0}=P^+\cap B $ are also top-dimensional. Moreover, since the only new pieces of boundaries of $V_i\subseteq U_i$  are introduced on the chord $\partial P^+\cap B$ and on the sphere $\partial B$, if the condition  that for  all $\sigma \subseteq [n]$, $\bigcap_{i\in \sigma}\partial U_i \subseteq \partial \left(\bigcap_{i\in \sigma} U_i\right)$  holds then the same condition should hold for the sets $V_i$. 
\end{proof}

A consecutive application of the above lemma to all the codewords in $\D\setminus \C$ for  any supra-code  $\D$ with the same simplicial complex yields Lemma \ref{lem:Ball:rules}.

\begin{proof}[Proof of Lemma \ref{lem:Ball:rules}]  Let $\U = \{U_i\}$ be an open convex cover in $\mathbb R^d$, $d\geq 2$,  with  $\code\left(\mathcal U,X\right)=\C$. Assume that  there exists an 
open Euclidean ball $B\subset \mathbb R^d $ such that $\code(\{B\cap U_i\},B\cap X) = \C$, and for every maximal set  $\alpha \in M(\C)$, its atom has non-empty intersection with the $(d-1)$-sphere  $\partial B$.   Let $\C  \subsetneq \D\subseteq \Delta(\C) $ and denote $\D\setminus \C=\{\sigma_1,\sigma_2,\dots,\sigma_l\}$. 
Let $\sigma_1\subsetneq  \alpha \in M(\C)$.  Because $\alpha\in M(C)$,  $A^{\U}_{\alpha}=\cap_{i\in \alpha}U_i$  is open, and thus 
 $\partial B \cap  \Cl \left( A^{\U}_{\alpha} \right)$  is relatively open in  $\partial B$.  
 We can now apply Lemma \ref{lemma:chord-cutter} to the ``missing'' codeword $\sigma_1$, and obtain a new cover $\V^{(1)}$ that again satisfies the condition of Lemma \ref{lemma:chord-cutter}. Consecutively applying   Lemma \ref{lemma:chord-cutter} with $\sigma_0=\sigma_j$, $j=2,3,\dots l$, we obtain covers $\V^{(j)}$, so that the last cover, $\V \od \V^{(l)}$ is the desired cover of Lemma  \ref{lem:Ball:rules}.
\end{proof}

%
%

\subsection{A closed convex realization for an intersection-complete code.}\label{apendix:IC}
Here we provide an  explicit  construction of a closed convex cover of an intersection-complete code.   
Intersection-complete codes are max intersection-complete, and thus Theorem \ref{t:mic} ensures that intersection-complete codes are both open convex and closed convex.  Nevertheless, a different construction below may be useful for  applications due to its simplicity. 

\begin{definition}\label{d:potential} The {\it potential cover} of the code $\cC$, is  a  collection  $\mathcal V  = \{V_i\}_{i \in [n]}$ of closed convex sets  $V_i \subset \R^{|\cC|}$, defined as follows. 
For each non-empty codeword $\sigma \in \cC$ let $e_\sigma$ be a unit vector in $\R^{|\cC|}$ so that $\{e_\sigma\}$ is a basis for $\R^{|\cC|}$.
For each $i \in [n]$, we define $V_i $ as the  convex hull
\[ V_i \od \conv \{ e_\sigma \;|\; \sigma \in \cC,\; \sigma\ni i\}.\]
\end{definition}
\noindent Since this is a cover by convex closed sets, the code of the potential cover is closed convex. Note however, that this cover is {\it not}  non-degenerate (Definition \ref{dfn:nondegenerate}), and  cannot be easily extended to an open convex cover.  
\begin{lemma} \label{l:pot_ic} Let $\mathcal V=\{V_i\}$ denote the potential cover of $\C$, and $X\od   \conv \{ e_\sigma \;|\; \sigma \in \cC,\; \sigma\neq \varnothing \}$. 
Then the code of the potential cover of $\C$   is the intersection completion of that code: 
  $\code(\mathcal V,X) = \widehat{\cC}$.
\end{lemma}
\begin{proof} 
 Note that because the vectors $e_\sigma$ are linearly independent, 
$$  \varnothing \not \in \widehat C \;\iff\; \exists i\in [n], V_i=X
 \; \iff\;  X=\bigcup_{i\in [n]} V_i
\; \iff\; \varnothing \not \in \code(\mathcal V,X).
$$
Moreover,   
\begin{equation}\label{eq:intersectionPi} \bigcap_{i\in \sigma} V_i=\conv \left \{ e_\tau \vert \, \tau\in \cC, \, \tau \supseteq \sigma \right \},
\end{equation}
in particular, $ \code\left ( \V, X \right ) \subseteq \Delta(\cC)$.
To show that $\code\left ( \V, X \right ) \subseteq \widehat{\cC}$, assume that a non-empty $\sigma \in  \code\left ( \V, X \right )$, i.e.\    $A^{\V}_\sigma=\left(\bigcap_{i \in \sigma}V_i \right) \setminus \bigcup_{j \notin \sigma}V_j$ is non-empty. 
If there exists an index  $j \in \left(\bigcap_{\sigma \subseteq \tau\in \C } \tau \right) \setminus \sigma,$ 
then by  \eqref{eq:intersectionPi}, 
 $\bigcap_{i \in \sigma}V_i \subset V_j$, which contradicts $\sigma \in \code\left ( \V, X \right ) $.
Hence $\sigma = \bigcap_{\cC\ni \tau\supseteq \sigma}\tau  \in \widehat{\cC}$. Conversely, assume that a non-empty $\sigma\in \widehat{\cC}$ and let $\sigma_1, \dots, \sigma_k \in \cC$ be code elements such that $\sigma = \bigcap_{\ell = 1}^k \sigma_\ell$.
Then the point $\frac{1}{k} \sum_{\ell = 1}^k e_{\sigma_\ell} \in \left( \bigcap_{i \in \sigma}V_i \right) \setminus \bigcup_{j \notin \sigma} V_j$.
Hence $\sigma \in \code\left ( \V, X \right )$.
\end{proof}

\noindent An immediate corollary is that any intersection-complete code is a closed convex code.\\

\section*{Acknowledgments}
VI was supported by Joint NSF DMS/NIGMS grant  R01GM117592,   NSF IOS-155925, and 
 ARO award W911NF-15-1-0084 for DARPA. CG  and VI also gratefully acknowledge the support of  SAMSI, under grant NSF DMS-1127914. We also  thank Carina Curto for numerous discussions. 

\bibliography{ConvexCodesRefs}
\bibliographystyle{plain}

\end{document}